\documentclass[12pt,centertags,oneside]{article}
\usepackage{amsmath,amstext,amsthm,amscd,typearea}
\usepackage{amssymb}
\usepackage{a4wide}
\usepackage[mathscr]{eucal}
\usepackage{mathrsfs}
\usepackage{typearea}
\usepackage{charter}
\usepackage{pdfsync}
\usepackage{url}
\usepackage{bigints}

\usepackage{mathtools}

\usepackage{xcolor}

\usepackage[a4paper,width=16.2cm,top=3cm,bottom=3cm]{geometry}

\usepackage{bbm, dsfont}

\numberwithin{equation}{section}

\newtheorem{theorem}{Theorem}[section]
\newtheorem{definition}[theorem]{Definition}
\newtheorem{proposition}[theorem]{Proposition}

\newtheorem{lemma}[theorem]{Lemma}
\newtheorem{remark}[theorem]{Remark}

\newtheorem{example}[theorem]{Example}

\newcommand{\supp}{{\rm Supp}}

\newcommand{\vol}{\mathop{\mathrm{vol}}}

\newcommand{\ddc}{dd^c}

\newcommand{\defeq}{\vcentcolon=}

\newcommand{\PSH}{{\rm PSH}}

\newcommand{\C}{\mathbb{C}}

\newcommand{\N}{\mathbb{N}}

\newcommand{\R}{\mathbb{R}}

\title{\bf Singularities of currents of full mass intersection}
\providecommand{\keywords}[1]{\textbf{\textit{Keywords:}} #1}
\providecommand{\subject}[1]{\textbf{\textit{Mathematics Subject Classification 2010:}} #1}

\author{Shuang Su}
\newcommand{\Addresses}{{
                \bigskip
                \footnotesize
                \textsc{Shuang Su, University of Cologne, Division of Mathematics, Department of Mathematics and Computer Science, Weyertal 86-90, 50931, K\"oln,  Germany}
                \noindent
                \par\nopagebreak
                \noindent
                \textit{E-mail address}: \texttt{ssu1@uni-koeln.de}        
}}

\date{\today}
\begin{document}
\maketitle
\begin{abstract}
    %Let $T_{j},T'_{j}$ be closed positive $(1,1)$-currents on a compact K\"ahler manifold that belong to the same cohomology class, for $1 \leq j \leq m$. In this paper

   In this paper, we study currents that have full mass intersection with respect to given currents in the mixed setting on a compact K\"ahler manifold. We compare their singularities by using Lelong numbers. Our main theorems generalize some results of Vu.

    %Let $T_{j},T_{j}'$ be closed positive $(1,1)$-currents in the same cohomology class on a compact K\"ahler manifold of dimension $n$, such that $T_{j}$ is less singular than $T'_{j}$ for $1 \leq j \leq m$, where $m \leq n$. In this paper, we compare the singularities of $T_{j}$ and $T'_{j}$ using Lelong numbers, in the case where $T'_{1}, \dots,T'_{m}$ have full mass intersection with respect to $T_{1}, \dots,T_{m}$.

    %In this paper, we study the Lelong numbers of currents $T'_{j}$ that are more singular than $T_{j}$ and of full mass intersection with respect to $T_{j}$.

\end{abstract}

%\tableofcontents

\noindent
%, {density current}, {tangent current}
, {Lelong number}, {full mass intersection}.
\\

\noindent
\subject{32U15}, {32Q15}.

\section{Introduction}
     Let $(X,\omega)$ be a compact K\"ahler manifold of dimension $n$. 
     %we denote by $\{R\}$ the cohomology class of $R$, and we denote by 
     %\[
     %   \|R\| \defeq\int_{X} R \wedge \omega^{n-p}
     %\]
     %the mass of $R$. 
     Let $\alpha,\beta \in H^{p,p}(X,\R)$. We say $\alpha \geq \beta$ if there exists a closed positive $(p,p)$-current $R \in \alpha - \beta$. Given a closed real $(p,p)$-current $R$, we denote by $\{R\} \in H^{p,p}(X,\R)$ the cohomology class of $R$. Let $\theta$ be a smooth closed real $(1,1)$-form on $X$. We denote by $\PSH(X,\theta)$ the set of $\theta$-plurisubharmonic functions on $X$ ($\theta$-psh in short). Let $u \in \PSH(X,\theta)$, we denote by $\theta_{u} \defeq \ddc u + \theta$, which is a closed positive $(1,1)$-current on $X$. 
     %Let $\{\theta\}$ be the cohomology class of $\theta$. 
     We say 

     \begin{itemize}
         \item[(1)] $\{\theta\}$ is K\"ahler if there exists a K\"ahler form in $\{\theta\}$.

         \item[(2)] $\{\theta\}$ is big if there exists $u \in \PSH(X,\theta)$ such that $ \theta_{u} =\ddc u + \theta \geq \epsilon \omega$ for some $\epsilon >0$. The current $\theta_{u}$ is called a K\"ahler current.
        
        \item[(3)] $\{\theta\}$ is pseudoeffective if $\PSH(X,\theta)$ is non empty. This means that there exists a closed positive $(1,1)$-current $T \in \{\theta\}$.
     \end{itemize}
     
     Let $\{\theta_{1}\}, \dots, \{\theta_{m}\}$ be pseudoeffective classes on $X$, $1 \leq m \leq n$. Let $T_{j},T_{j}' \in \{\theta_{j}\}$ be closed positive $(1,1)$-current, $1 \leq j \leq m$, such that $T_{j}$ is less singular than $T'_{j}$, meaning that the potential of $T_{j}$ is greater than the potential of $T'_{j}$, modulo an additive constant. The following monotonicity property (see \cite{BEGZ,Lu-Darvas-DiNezza-mono,Viet-generalized-nonpluri,WittNystrom-mono}) of the non-pluripolar product plays a crucial role in pluripotential theory:
     \[
        \big{\{}\big{\langle} \wedge_{j=1}^{m} T'_{j}  \big{\rangle} \big{\}} \leq \big{\{} \big{\langle} \wedge_{j=1}^{m} T_{j} \big{\rangle}\big{\}},
     \]
     %and has many application, see \cite{Su-Vu_volumepaper}. 
     Here, $\langle \cdot \rangle$ denotes the non-pluripolar product (see Section \ref{nonpluripolarsection}). When the equality holds, we say $T'_1, \dots ,T'_m$ have full mass intersection with respect to $T_1, \dots , T_m$. The aim of this paper is to investigate the singularities of $T_{j}$ and $T'_{j}$ when $T'_{1}, \dots, T'_{m}$ have full mass intersection with respect to $T_{1}, \dots , T_{m}$.

     Let $T $ be a closed positive $(1,1)$-current on $X$. The Lelong number is a notion measuring the singularities of $T$ and is defined as follows. Let $x \in X$, and write $T=\ddc \varphi$ around $x$, where $\varphi$ is a plurisubharmonic (psh) function. The Lelong number of $T$ at $x$ is denoted and defined by 
     \[
        \nu(T,x) \defeq \max \{ \gamma \in \R^{+} | \varphi(z) \leq \gamma\log |z-x|+ O(1) \text{ near $x$}\}.
     \]
     Let $V$ be an analytic subset of $X$. By the well-know theorem of Siu on Lelong numbers in \cite{Siu}, we have 
     \[
        \nu(T,x) = \min_{x' \in V} \{\nu(T,x')\},
     \]
     for $x \in V$ outside some proper analytic subset of $V$. The right-hand side of the above equation is called the generic Lelong number of $T$ along $V$ and is denoted by $\nu(T,V)$.
       
     Let $\{ \theta \}$ be a big class, we say a closed positive current $T \in \{\theta\}$ is big if $\int_{X} \langle T^{n} \rangle > 0$. A characterization of the bigness of $T$ is provided in \cite[Proposition 3.6]{Darvas_Xia-volumelinebundle}, which states that there exists a K\"ahler current $P \in \{ \theta \}$ that is more singular than $T$. Let $u \in \PSH (X, \theta)$. The $\mathcal{I}$-model envelope $P^{\theta}[u]_{\mathcal{I}} \in \PSH(X,\theta)$ was introduced and studied in \cite{DarvasXiaconfigu}, and is defined as
     \[
        P^{\theta}[u]_{\mathcal{I}} \defeq  (\sup \{ w \in \PSH(X,\theta) | w \leq 0, \mathcal{I}(tw) \subset \mathcal{I}(t u)), t \geq 0\})^{*},
     \]
     where $(\cdot)^{*}$ denotes the upper-semicontinuous envelope, and $\mathcal{I}(tu)$ denotes the multiplier ideal sheaf, locally generated by holomorphic functions $f$ such that $|f|^{2} e^{-t u}$ is integrable. 
     %We say $u \in \PSH(X,\theta)$ is $\mathcal{I}$-model if $u = P^{\theta}[u]_{\mathcal{I}}$, 

     We say a closed positive $(1,1)$-current $T = \ddc u + \theta$ is $\mathcal{I}$-model if $u = P^{\theta}[u]_{\mathcal{I}}$.
     %Let $T=\ddc u + \theta$ be a closed positive $(1,1)$-current. We say that $T$ has $\mathcal{I}$-model potential singularity type if there exists $v \in \PSH(X,\theta)$ such that $v = P^{\theta}[v]_{\mathcal{I}}$ and 
     %\[
     %   v + \mathcal{O}(1) \leq u \leq v + \mathcal{O}(1).
     %\]
     The following is our main result. We consider the case where $T_{1} , \dots , T_{m}$ are big and $\mathcal{I}$-model (this condition can be slightly relaxed, see Remark \ref{relaxrmk}), and we compare the Lelong numbers of $T_{j}$ and $T'_{j}$ under the full mass intersection assumption.

     %We provide a necessary condition for $T'_{1}, \dots, T'_{m}$ to be full mass intersection with respect to $T_{1}, \dots ,T_{m}$. 

    %Let $X$ be a Kähler manifold of dimension $n$, and $\omega$ be a Kähler form. Let $\alpha_1, \dots, \alpha_m$ be pseudoeffective classes, $1 \leq m \leq n$. 
    %For each $1 \leq j \leq m$, let $T_j, T'_j$ be closed positive $(1,1)$-currents in $\alpha_j$, where $T'_j$ is less singular than $T_j$. By the monotoncity property of non-pluripolar product, the following holds 
    %\[
    %    \{\langle T_1 \wedge \dotsi \wedge T_m \rangle \} \leq \{ \langle T'_1 \wedge \dotsi \wedge T'_m \rangle\}.
    %\]
    %When the equality holds, we say $T_1, \dots ,T_m$ are of relative full mass intersection with respect to $T'_1, \dots , T'_m$. In this paper, we try to compare the singularity types of $T_j$ and $T'_j$ when the equality above occurs. 

    %Let $\alpha$ be a pseudoeffective class, and $T$ be a closed positive $(1,1)$-current in $\alpha$. $T$ is said to be big if $\int_{X} \langle T^n \rangle >0$. The following is our first main result.
    
    %For the case $T_j$ is of minimal singularities, there are some results (see \cite{Vu_lelong-bigclass, Vu_lelong-bignef-quantitative}). The following is our first main result.

    \begin{theorem}
    \label{maintheorem}
        Let $\{\theta_{1}\}, \dots ,\{\theta_{m}\}$ be big classes, and let $T'_{j}, T_{j} \in \{\theta_{j}\}$ be closed positive $(1,1)$-currents such that
        \begin{itemize}
            \item[(1)]
                $T_{j}$ is big and is $\mathcal{I}$-model.
            \item[(2)]
                $T_{j}$ is less singular than $T'_{j}$.
        \end{itemize}
        Let $V$ be a proper irreducible analytic subset such that $\dim(V) \geq n-m$. If $T'_1, \dots , T'_m $ are of full mass intersection with respect to $T_{1}, \dots , T_{m}$. Then, there exists $1 \leq j \leq m$ such that $\nu(T'_{j},V)=\nu (T_j,V)$. 
    \end{theorem}

    Theorem \ref{maintheorem} was proved in \cite{GZ-weighted} for the case that $\theta_{1} = \dotsi =\theta_{n} = \omega$ is a K\"ahler form, $T_{1}'= \dotsi = T_{n}'$ and $T_{1} = \dotsi = T_{n} \in \{\omega\}$ is of minimal singularities. This result was later generalized in \cite{Vu_lelong-bigclass} to the setting where $\{\theta_{1}\}, \dots , \{\theta_{m}\}$ are big classes, and each $T_{j} \in \theta_{j}$ is of minimal singularities. For the self-intersection case ($T_{1}'= \dotsi = T_{n}'$, $T_{1} = \dotsi=T_{n}$), a characterization of currents having full mass intersection was given in \cite{Lu-Darvas-DiNezza-mono} in terms of the singularity types of the envelopes of the given currents.

    When $m=n$, we have the following more general quantitative result.

    %Induced by the proof of Theorem \ref{maintheorem}, we get the following more general quantitative result.

    \begin{theorem}
    \label{maintheroem2}
    Let $\mathscr{B}$ be a closed cone in the cone of big classes. There exists a constant $C>0$ only depends on the manifold $X$ and the cone $\mathscr{B}$ such that for every $x_{0} \in X$, $\{\theta_{j}\} \in \mathscr{B}$ and $T'_{j}, T_{j} \in \{\theta_{j}\}$, $1 \leq j \leq n$, satisfied the conditions in Theorem \ref{maintheorem}.
    %\begin{itemize}
    %    \item[(1)] $T'_{j}$ is big and $\mathcal{I}$-model.
    %    \item[(2)] $T_{j}$ is more singular than $T'_{j}$.
    %\end{itemize}
    The following inequality holds
    \begin{equation}
    \label{inequa1.2}
        \int_{X} \Big{(} \big{\langle} \wedge_{j=1}^{n} T_{j} \big{\rangle} - \big{\langle} \wedge_{j=1}^{n}T'_{j} \big{\rangle}  \Big{)}\geq C \prod_{j=1}^{n} \big{(} \nu (T'_{j},x_{0})-\nu(T_{j},x_{0})\big{)}.
    \end{equation}
    \end{theorem}

    We note that our proof strategy is based on \cite{Vu_lelong-bigclass}, where we generalize Theorems 1.1 and 1.2 from that paper. 
    %In \cite{Vu_lelong-bigclass}, these results deals with the case where $T$ has minimal singularities.

    The organization of this paper is as follows. In section 2, we recall the singularity types of psh function and their envelopes. We also review the non-plruipolar product and some of its important properties.
    %Next, we recall the pseudo-metric on the space of singularity types and discuss some properties of both the non-pluripolar product and the pseudo-metric. 
    We conclude this section by recalling Demailly's analytic approximation theorem. Finally, Theorem \ref{maintheorem} and Theorem \ref{maintheroem2} are proved in section 3.\\

    \noindent
    \textbf{Acknowledgments.} 
    I am grateful to my advisor, Duc-Viet Vu, for his valuable discussions. This research is partially funded by the Deutsche Forschungsgemeinschaft (DFG, German Research Foundation)-Projektnummer 500055552 and by the ANR-DFG grant QuaSiDy, grant no ANR-21-CE40-0016.

\section{Preliminaries}

\subsection{Singularity type and Envelopes}
    
    Let $\{\theta\}$ be a pseudoeffective class on $X$. Let $u,v \in \PSH(X, \theta)$. We say $u$ is less singular (more singular) than $v$ if there exists a constant $C>0$ such that $v \leq u +C$ ($u \leq v +C$). This gives us a equivalence relation on $\PSH(X,\theta)$, and we denote by $[u]$ the singularity type of $u$, which is the equivalence class contains $u$. We denote by $[v] \preceq [u]$ if $u$ is less singular than $v$. Let $T_{1} = \ddc u_{1}+ \theta$ and $T_{2} = \ddc u_{2} + \theta$ be two closed positive $(1,1)$-currents in the same cohomology class. We say that $T_{1}$ is less singular than $T_{2}$ if $[u_{1}] \succeq [u_{2}]$, and we denote this by $T_{1} \succeq T_{2}$. Understanding singularity type of $\theta$-psh function plays an important role in pluripotential theory, see \cite{Lu-Darvas-DiNezza-mono,Darvas-Lu-DiNezza-singularity-metric, DarvasXiaconfigu}.
    
   %We denote by $S(X,\theta)$ the space of singular types of $\theta$-psh functions. 

    %To study of the full mass intersection problem, 
    We recall the envelopes of psh function (current) in this section, following the construction in \cite{DarvasXiaconfigu}. We refer to \cite{Lu-Darvas-DiNezza-logconcave, Darvas-Lu-DiNezza-singularity-metric, Darvas_Xia-volumelinebundle} for further results.

   % We say $u,v$ have the same singularities type if there exists a constant $C>0$ such that $u-C \leq v \leq u+C$.  Given two classes $\alpha, \beta$. We say $\alpha$ is less singular than $\beta$ ($\beta  \preceq \alpha$) if for $u \in \alpha$ and $v \in \beta$, there exists $C>0$ such that $v \leq u +C$. Given $u \in \PSH(X,\theta)$. We For convenience, we denote the current $\ddc u + \theta$ by $\theta_{u}$.
    %Notice that the definition is independent of the choice of $u \in \alpha$ and $v \in \beta$.
    
    \begin{definition}
        Let $\{\theta\}$ be a pseudoeffective class, and let $u \in \PSH(X, \theta)$. We define the upper envelop and $\mathcal{I}$-envelope of $u$ as
        \begin{align*}
            &P^{\theta}[u] \defeq (\sup\{w \in \PSH(X,\theta) | [w] \preceq [u], \quad w\leq 0\})^{*}\in \PSH(X,\theta) \text{ and}\\
            & P^{\theta}[u]_{\mathcal{I}} \defeq  (\sup \{ w \in \PSH(X,\theta) | w \leq 0, \quad \mathcal{I}(tw) \subset \mathcal{I}(t u)), \quad t \geq 0\})^{*}  \in \PSH(X,\theta),
        \end{align*}
        respectively. Note that $P^{\theta}[u],P_{\mathcal{I}}^{\theta}[u]$ only depends on the singularity type of $u$.

        Let $T = \ddc u + \theta$ be a closed positive $(1,1)$-current current, where $u \in \PSH(X,\theta)$. The upper envelope and $\mathcal{I}$-envelope of $T$ are defined as 
        \begin{align*}
            &P[T] \defeq \theta_{P^{\theta}[u]} =\ddc P^{\theta}[u] +\theta \text{ and}\\
            &P[T]_{\mathcal{I}} \defeq \theta_{P^{\theta}[u]_{\mathcal{I}}} =\ddc P^{\theta}[u]_{\mathcal{I}} + \theta,
        \end{align*}
        respectively.
    \end{definition}

    \begin{remark}
    \label{welldefinedenvelope}
        The envelopes $P[T]$ and $P[T]_{\mathcal{I}}$ of $T$  are independent of the choice of $\theta$ and $u \in \PSH(X, \theta)$.
    \end{remark}
    \begin{proof}
        let $\theta'$ be a smooth closed real $(1,1)$-form cohomologous to $\theta$. Write $T = \ddc u'+ \theta'$, where $u' \in \PSH(X,\theta')$. By the $\partial \bar{\partial}$-lemma, there exists smooth function $f \leq 0$ on $X$ such that $\ddc f = \theta - \theta'$.
        %Write $T$ as $\ddc \tilde{u} +\tilde{\theta}$ for some $\tilde{\theta}$ cohomologous to $\theta$ and $\tilde{u} \in \PSH(X, \tilde{\theta})$. By  
        Now, we have 
        \begin{align*}
            T &= \ddc u + \theta \\
            &= \ddc (u+ f) + \theta'\\
            &= \ddc u' + \theta'.
        \end{align*}
        This shows that $[u+f]=[u']$, and we get
        %Since $\ddc (u+f - \tilde{u})=0$,  $u+f - \tilde{u}$ is a pluriharmonic function, which is bounded. 
        \begin{align*}
            \ddc P^{\theta}[u] + \theta &= \ddc( P^{\theta}[u]+f) + \theta'\\
            &= \ddc P^{\theta'}[u+f] + \theta'\\
            & = \ddc P^{\theta'}[u']+ \theta'.
        \end{align*}
        Therefore, we get the well-definedness of $P[T]$. The proof of the well-definedness of $P[T]_{\mathcal{I}}$ is similar.
    \end{proof}

    %Let $u \in \PSH(X,\theta)$, the singularity type $[u]$ is called an $\mathcal{I}$-model singularity type if $[u]= [P^{\theta}[u]_{\mathcal{I}}]
    %]$
    Let $u \in \PSH(X,\theta)$. We say $u$ is $\mathcal{I}$-model if $u= P^{\theta}[u]_{\mathcal{I}}$. Similarly, we say the closed positive $(1,1)$-current $T$ is   $\mathcal{I}$-model if $T= P[T]_{\mathcal{I}}$. Here ,we give some examples of $\mathcal{I}$-model $\theta$-psh functions.

    \begin{example}
    \mbox{}
        \begin{itemize}
            \item[(1)] 
            Set $V_{\theta} := (\sup \{ v \in \PSH(X,\theta)| v \leq 0\})^{*}$, which is of minimal singularities in $\PSH(X,\theta)$. Since $V_{\theta}$ is less singular than all $\theta$-psh functions, it is $\mathcal{I}$-model.

            \item[(2)] %Let $u \in \PSH(X,\theta)$ be of analytic singularities. Then, $P^{\theta}[u]= P^{\theta}[u]_{\mathcal{I}}$ 
            Let $u \in \PSH(X,\theta)$. By \cite[Proposition 2.18]{DarvasXiaconfigu}, we have 
            \[
                P^{\theta}[P^{\theta}[u]_{\mathcal{I}}]_{\mathcal{I}}= P^{\theta}[u]_{\mathcal{I}}.
            \]
            Hence, $P^{\theta}[u]_{\mathcal{I}}$ is $\mathcal{I}$-model.

        \end{itemize}
        
    \end{example}

    Let $f \colon X \longrightarrow [-\infty,\infty]$ be a function on $X$. The envelope with respect to $f$ is defined as
    \[
        P^{\theta}(f) := (\sup\{v \in \PSH(X,\theta) | v \leq f\})^{*},
    \]
    which was introduced and studied in \cite{Darvas-Lu-DiNezza-singularity-metric}.

    \begin{definition}
    \label{differenceofcurrents}    
        %Let $\{\theta_{1}\}, \{\theta_{2}\}$ be pseudoeffective classes, let $u_{1}\in\PSH(X,\theta_{1})$ and $u_{2}\in\PSH(X,\theta_{2})$. If there exists $v \in \PSH(X,\theta_{1}-\theta_{2})$ such that $[v+u_{2}] \preceq [u_{1}]$, then we define the difference envelope of $u_{1}$ and $u_{2}$ as 
        %\[
        %    P(u_{1} - u_{2}) \defeq \big{(}\sup\{v \in \PSH(X,\theta_{1}-\theta_{2}) | v \leq 0,\quad [v+u_{2}] \preceq [u_{1}]\}\big{)}^{*} \in \PSH(X,\theta_{1} - \theta_{2}).
        %\]
        Let $T_{1}=\ddc u_{1} + \theta_{1},T_{2}= \ddc u_{2}+ \theta_{2}$ be closed positive $(1,1)$-currents, where $u_{1} \in \PSH(X,\theta_{1}),u_{2} \in \PSH(X,\theta_{2})$.
        We define the difference envelope of $T_{1}$ and $T_{2}$ as
        \[
            P(T_{1}-T_{2}) \defeq \ddc P^{\theta_{1}-\theta_{2}}(u_{1}-u_{2})+ (\theta_{1}-\theta_{2}).
        \]
        Here, we note that if $u_{1}(x)=u_{2}(x)=-\infty$ for $x \in X$, then we define $(u_{1}-u_{2})(x) = - \infty$.
    %\mbox{}
    %\begin{itemize}
    %    \item[(1)] 
    %        Let $\{\theta_{1}\}, \{\theta_{2}\}$ be pseudoeffective classes, and $T_{1} \in \{\theta_{1}\}$, $T_{2} \in \{\theta_{2}\}$ be closed positive $(1,1)$-currents. Let $u_{1} \in \PSH(X, \theta_{1})$, $u_{2} \in \PSH(X,\theta)$ such that 
    %        \[
    %            T_1 =\ddc u_1 + \theta_1 \mbox{, }T_2 = \ddc u_2 + \theta_2.
    %        \]
    %        We define the enveloping difference of $T_1$ and $T_2$ to be $P(T_1 -T_2) \defeq \ddc u_{1-2} + (\theta_1 -\theta_2)$, where 
    %        \[
    %            u_{1-2} = \big( \sup \{ v \in \PSH(X, \theta_{1}-\theta_{2}) | [v+u_2] \preceq [u_1] \} \big)^{*}.
    %        \]
    %    \item [(2)]
    %        Let $\{ \theta\}$ be a pseudoeffective class, and $T \in \{\theta\}$ be a closed positive $(1,1)$-current. Write $T$ as $\ddc u + \theta$ for some $u \in \PSH(X,\theta)$. Let $\gamma$ be a smooth closed $(1,1)$-form. We define the enveloping sum of $T$ and $\gamma$ to be $P(T + \gamma) \defeq \ddc \tilde{u} + (\theta + \gamma)$, where 
    %        \[
    %            \tilde{u} = (\sup \{ v \in \PSH(X, \theta + \gamma) | v \in \PSH(X, \theta), [v] \preceq [u]\})^{*}.
    %        \]
    %\end{itemize}
    Similar to Remark \ref{welldefinedenvelope}, one can check that $P(T_1 -T_2)$ is independent of the choices of $\theta_{j}$ and $u_{j}$ for $j=1,2$.
    \end{definition}

\subsection{Relative non-pluripolar products}
\label{nonpluripolarsection}
    We recall the non-pluripolar product and some of its properties in this section. Non-pluripolar product is a notion of wedge product of closed positive $(1,1)$-currents. It was introduced in \cite{BT_fine_87, BEGZ, GZ-weighted}, and was generalized to the relative setting in \cite{Viet-generalized-nonpluri}. We explain briefly how to do it.
     
    Let $(X,\omega)$ be a compact Kähler manifold of dimension $n$. Let $T_1, \dots, T_m$ be closed positive $(1,1)$-currents, and let $T$ be a closed positive $(p,p)$-current, where $p+m \leq n$. For $1 \leq j \leq m$, write $T_{j} = \ddc u_{j}+ \theta_{j}$, %where $\theta_{j}$ is a smooth closed $(1,1)$-form and 
    where $u_{j} \in \PSH(X,\theta_{j})$. Set 
    \[
        R_k \defeq \mathbbm{1}_{\cap_{j=1}^{m}\{u_j >-k\}}\wedge_{j=1}^{m}(\ddc \max\{u_j,-k\}+ \theta_j) \wedge T,
    \]
    which is positive for each $k$ , see\cite[Lemma 3.2]{Viet-generalized-nonpluri}. By \cite[Lemma 3.4]{Viet-generalized-nonpluri}, one can see that the mass of $R_{k}$ is uniformly bounded in $k$ and $\{R_{k}\}_{k}$ converges weakly to a closed positive current, denoted by $\langle T_1 \wedge \dotsi T_m \dot{\wedge} T \rangle$, which is called the non-pluripolar product relative to $T$ of $T_1, \dots, T_m$. 
    
    Let $T$ be a closed positive $(p,p)$-current, the mass of $T$ is defined and denoted as 
    \[
        \|T\| \defeq \int_{X} T \wedge \omega^{n-p}.
    \]
    Let $P$ be a closed positive $(1,1)$-current. We denote by $I_{P}$ the pole of $P$, which is the set of $x \in X$ such that the potential of $P$ equal to $- \infty$ at $x$. A set $A \subset X$ is complete pluripolar if $A = \{x | u(x) = - \infty\}$ for some quasi-psh function $u$ on $X$. The following are some basic properties of the relative non-pluripolar product, derived from \cite[Proposition 3.5]{Viet-generalized-nonpluri}.

    \begin{proposition}
    \label{pro-sublinearnonpluripolar} 
    \mbox{}
        \begin{itemize}
            \item[(1)]
                For $R:= \langle \wedge_{j=l+1}^m T_j \dot{\wedge} T \rangle$, we have $\langle \wedge_{j=1}^m T_j \dot{\wedge} T \rangle = \langle \wedge_{j=1}^l T_j \dot{\wedge} R \rangle$.

            \item[(2)]
                For every complete pluripolar set $A$, we have 
                $$\mathbbm{1}_{X \backslash A}\langle \wedge_{j=1}^{m} T_{j}\dot{\wedge} T\rangle= \big\langle \wedge_{j=1}^{m}T_{j}  \dot{\wedge}(\mathbbm{1}_{X \backslash A} T)\big\rangle.$$
                In particular, the equality
                $$\langle \wedge_{j=1}^m T_j \dot{\wedge} T \rangle = \langle \wedge_{j=1}^m T_j \dot{\wedge} T' \rangle$$
                holds, where $T'\defeq \mathbbm{1}_{X \backslash \cup_{j=1}^m I_{T_j}} T$.
        \end{itemize}
    \end{proposition}
    
    %Let $\{\theta\}$ be a pseudoeffective class, and let $T_{1}, T_{2} \in\{\theta\}$ be closed positive $(1,1)$-currents. We write $T_{1}=\ddc u_{1} + \theta$ and $T_{2}=\ddc u_{2}+\theta$ respectively, where $u_{1},u_{2} \in \PSH(X,\theta)$. We say $T_{1}$ is less singular (more singular) than $T_{2}$ if $[u_{1}] \succeq [u_{2}]$ ($[u_{1}] \preceq [u_{2}]$), denoted by $T_{1} \succeq  T_{2} $ ($T_{1} \preceq T_{2})$. 
    The relative non-pluripolar product has the following crucial monotonicity property.

    \begin{theorem}[\cite{Viet-generalized-nonpluri}, Theorem 4.4]
    \label{monotonicity}
        Let $T'_j, T_j$ be closed positive $(1,1)$-currents on $X$, $1 \leq j \leq m$, such that $T'_j,T_j$ are in the same cohomology class and $T_j$ is less singular than $T'_j$. Let $T$ be closed positive $(p,p)$-current on $X$ such that $p+m \leq n$. Then, we have 
        \[
            \big{\{\langle} \wedge_{j=1}^{m} T'_{j} \dot{\wedge}T \big{\rangle\}} \leq \big{\{\langle} \wedge_{j=1}^{m} T_{j} \dot{\wedge} T \big{\rangle\}}.
        \]
        Furthermore, if $T$ is of bi-degree $(1,1)$ and $T' \in \{T\}$ is a closed positive $(1,1)$-current which is more singular than $T$. Then, 
        \[
            \big{\{\langle} \wedge_{j=1}^{m} T'_{j} \dot{\wedge}T' \big{\rangle\}} \leq \big{\{\langle} \wedge_{j=1}^{m} T_{j} \dot{\wedge} T \big{\rangle\}}.
        \]
    \end{theorem}

    The following lemma will be used in the proof of our main theorem.

    \begin{lemma}
    \label{nonpluripolar_complete}
        Let $\{\theta_{1} \} , \dots , \{\theta_{m}\}$ be pseudoeffective classes, $1 \leq m \leq n$. For $1 \leq j \leq m$, let $T'_j, T_j \in \{\theta_{j}\}$ be closed positive $(1,1)$-currents such that $\supp(T'_j) \subset W_j$ for some complete pluripolar set $W_j$. Then, the following equality holds
        \[
            \langle \wedge_{j=1}^{m}(T_{j} + T'_{j}) \rangle = \langle \wedge_{j=1}^{m} T_{j} \rangle.
        \]
    \end{lemma}
    \begin{proof}
        By the multi-linearity of non-pluripolar product (\cite[Proposition 1.4]{BEGZ}), we have  
        \[
            \langle \wedge_{j=1}^{m} (T_{j} + T'_{j}) \rangle=\langle \wedge_{j=1}^{m} T_{j}\rangle + \sum_{s=1}^{m} R_{s},
        \]
        where $R_{s} = c_{j}\langle \wedge_{j=1}^{s} T'_{j} \wedge \wedge_{\ell = s+1}^{m} T_{\ell} \rangle$, $c_{j} \in \N^{+}$. 
        %$R_{s}$ is supported in $W$, for $1 \leq s \leq m$. On the other hands, 
        By the construction of non-pluripolar product, we see that $R_{s}$ is supported in %$\cap_{j=1}^{s} \supp(T'_{j}) \subset \cup_{j=1}^{s} W_j=W$, 
        $W = \cup_{j=1}^{n}W_{j}$, where $W$ is a complete pluripolar set. On the other hand, the non-pluripolar product of currents does not charge mass on complete pluripolar set. Therefore, we get $R_{s}=0$ for $1 \leq s \leq m$, which completes the proof.

        %Set $W =\cup_{j=1}^{m} W_j$, which is also a complete pluripolar set. By the definition of non-pluripolar product, we get $\langle R_{j,S} \wedge T'_j \rangle$ has no mas on $W$. On the other hands, since $\supp(T'_j) \subset W$,
        %\begin{align*}
        %     \mathbbm{1}_{X \backslash W} \langle R_{j,s} \wedge T'_{j} \rangle  &= \mathbbm{1}_{X \backslash (W \cup I_{T'_j})} \langle R_{j,S} \wedge T'_j \rangle \\
        %     &= 0.
        %\end{align*}
        %Therefore, we get $\langle R_{j,S} \wedge T'_j \rangle =0$ on $X$, and $\langle (T_1 + T'_1) \wedge \dotsi \wedge (T_m + T'_m) \rangle = \langle T_1 \wedge \dotsi \wedge T_m \rangle$.
    \end{proof}

   Let $T$ be a closed positive $(1,1)$-current on $X$.
   We say that $T$ has analytic singularities if its potential $\phi$ can be written as 
   \[
        c\log\sum_{j=1}^{m}|f_{j}|^{2}+g,
   \]
   where $c > 0$, $f_{j}$ are holomorphic functions, and $g$ is a smooth function on $X$. The lemma below generalizes \cite[Lemma 2.5, Lemma 3.1]{Vu_lelong-bigclass}.

    \begin{lemma}
    \label{replacementlemma2}
        Let $\{\theta_{1}\}, \dots ,\{\theta_{m}\}$ be pseudoeffective $(1,1)$-classes, $1 \leq m \leq n $. Let $T'_j , T_{j} \in \{\theta_j\}$ be closed positive $(1,1)$-currents such that $T_{j}$ is less singular than $T'_{j}$, $1 \leq j \leq m$. Set $T_{j}^{p} \defeq \mathbbm{1}_{I_{T_{j}}} T_{j}$. If there exists a K\"ahler current $P_{j} \in \{\theta_{j}\}$ with analytic singularities that is more singular than $T_{j}$. Then, $\mathbbm{1}_{I_{T'_{j}}}T'_j \geq T^{p}_{j}$ for $1 \leq j \leq m$, and the following holds. 
        \begin{align}
            \langle \wedge_{j=1}^{m} T'_{j} \rangle &= \langle \wedge_{j=1}^{m}(T'_j - T^{p}_{j}) \rangle,\label{replacementequality}\\     
            \|\langle \wedge_{j=1}^{m} T'_{j} \rangle\| &=  \| \langle \wedge_{j=1}^{m} P(T'_{j} - T^{p}_{j})  \rangle\|.\label{replacementnormequality}
        \end{align}
    \end{lemma}

    \begin{proof}
        The strategy of the proof is based on \cite{Vu_lelong-bigclass}. First, by applying Lemma \ref{nonpluripolar_complete}, we immediately obtain  (\ref{replacementequality}). We now show that $\mathbbm{1}_{I_{T'_{j}}}T'_{j} \geq T_{j}^{p}$ holds for $1 \leq j \leq m$. Since $T_{j}$ is less singular than $P_{j}$, we get $\supp(T^{p}_{j}) \subset I_{P_{j}}= V_{j}$, where $V_{j}$ is an analytic subset. By the first support theorem (\cite[Chapter \uppercase\expandafter{\romannumeral3}, Theorem 2.10]{Demailly_ag}), the analytic subset $V_{j}$ is of dimension $n-1$. Furthermore, the second support theorem (\cite[Chapter \uppercase\expandafter{\romannumeral3}, Theorem 2.13]{Demailly_ag}) implies that 
        \[
            T_{j}^{p}= \sum_{V_{j,k}} \lambda_{j,k} [V_{j,k}],
        \]
        where $\{V_{j,k}\}$ are the irreducible components of $V_{j}$. Since $T_{j}$ is less singular than $T'_{j}$, we get 
        \[
            \nu(T'_{j},V_{j,k}) \geq \nu(T_{j},V_{j,k}) \geq \nu(T^{p}_{j}, V_{j,k}) = \lambda_{j,k}
        \]
        for each irreducible component $V_{j,k}$. This shows that $T'_{j} \geq T_{j}^{p}$, and we get $\mathbbm{1}_{I_{T'_{j}}}T'_j \geq T^{p}_{j}$.

        Now, we prove (\ref{replacementnormequality}). We write  
        \[
            T'_{j} = (\theta_{j})_{u'_{j}} = \ddc u'_{j} + \theta_{j}, \quad T^{p}_{j} = (\theta^{p}_{j})_{u^{p}_{j}} =  \ddc u^{p}_{j} + \theta^{p}_{j},
        \]
        where $u_{j}' \in \PSH(X,\theta_{j})$, $u_{j}^{p} \in \PSH(X,\theta_{j}^{p})$. Recall that the difference envelope of $T_{j}'$ and $T_{j}^{p}$ is defined by $P(T_{j}'-T_{j}^{p}) = \ddc w_{j} + (\theta_{j}- \theta_{j}^{p})$, where 
        \[
            w_{j} \defeq (\sup \{ v_{j} \in \PSH(X, \theta_{j}-\theta^{p}_{j}) | v_{j} \leq u'_{j}-u_{j}^{p}    
            \})^{*}.
        \]
        Set $K_{j} = \{ v_{j} \in \PSH(X, \theta_{j}-\theta^{p}_{j}) | v_{j} \leq u'_{j}-u_{j}^{p} \}$.
        By Choquet's lemma, there exists a sequence $v_{j,k} \nearrow w_j$ a.e. in $K_{j}$. Hence, $v_{j,k} \longrightarrow v_{j}$ in capacity. By \cite[Theorem 2.3]{Lu-Darvas-DiNezza-mono}, we get
        \begin{align}
        \label{replacementlemma2eq}
            \| \langle \wedge_{j=1}^{m} P(T'_{j} - T^{p}_{j})\rangle\| &= \lim_{k \rightarrow \infty} \| \langle \wedge_{j=1}^{m} (\theta_{j} - \theta^{p}_{j})_{v_{j,k}}  \rangle \| 
        \end{align}
        By Lemma \ref{nonpluripolar_complete}, and Theorem \ref{monotonicity}, we then have 
        \begin{equation*}               (\ref{replacementlemma2eq})= \lim_{k \rightarrow \infty} \|\langle \wedge_{j=1}^{m}(\theta_{j})_{v_{j,k}+u^{p}_{j}} \rangle\|
            \leq \| \langle \wedge_{j=1}^{m}T'_{j}\rangle \|.
        \end{equation*}
        On the other hands, by using (\ref{replacementequality}) and Theorem \ref{monotonicity} again, we get 
        \begin{align*}
            \| \langle \wedge_{j=1}^{m} P(T'_{j} - T^{p}_{j}) \rangle\| &\geq \| \langle \wedge_{j=1}^{m} (T'_{j} - T^{p}_{j})\rangle \|
            = \| \langle \wedge_{j=1}^{m} T'_{j}  \rangle \|.
        \end{align*}
        This completes the proof.
    \end{proof}

\subsection{Demailly's analytic approximation theorem}
    In this section, we recall Demailly's analytic approximation theorem, introduced in \cite{Demailly_regula_11current, Demailly_analyticmethod}. Furthermore, we discuss the convergence of the Monge-Ampère masses of the currents in Demailly's approximation sequence.

    \begin{theorem} 
    \label{demaillyapprotheorem}
    (\cite[Corollary 14.13]{Demailly_analyticmethod})
        Let $\{\theta\}$ be a pseudoeffective class, and $u \in \PSH(X,\theta)$. Then, there exists a sequence $u_{k}^{D} \in \PSH(X, \theta + \epsilon_{k} \omega)$, where $\epsilon_{k}$ decreases to $0$, such that 
        \begin{itemize}
            \item[(1)] $u_{k}^{D} \ge u$ and $u_k^D$ converges to $u$ in $L^1$.
            \item[(2)] $u_{k}^{D}$ has analytic singularities.
            \item[(3)] $\nu(T_k,x)$ converges to $\nu(T,x)$ uniformly on $X$, where $T_k = \ddc u^{D}_k + (\theta + \epsilon_k \omega)$.
        \end{itemize}
    \end{theorem}

    Here, we briefly recall the construction of $u^{D}_{k}$.\\
    \noindent
    \textbf{Step 1. (Local approximation)}
        Let $\varphi$ be a psh function on some bounded pseudoconvex open set $\Omega \subset \C^{n}$. Set $\varphi_{k} \defeq \frac{1}{2k} \log \sum |\sigma_{\ell}|^{2}$, where $\sigma_{\ell}$ is an orthonormal basis of 
        \[
            \mathscr{H}_{\Omega}(k \varphi) \defeq \{f | \text{$f$ is holomorphic on $\Omega$, and $\int_{\Omega} |f|^{2} e^{-2k \varphi}$} < +\infty\}.
        \]
        By \cite[Theorem 14.2]{Demailly_analyticmethod}, there exists $C_{1},C_{2}>0$ independent of $k$ such that 
        \begin{itemize}
            \item[(1)] 
                $\varphi(z) - \frac{C_{1}}{k} \leq \varphi_{k}(z) \leq \sup_{|\xi-z| <r} \varphi(\xi) + \frac{1}{k} \log \frac{C_{2}}{r^{n}}$ for $z \in \Omega$ and $r < d(x , \partial \Omega)$. Hence, $\varphi_{k} \longrightarrow \varphi$ pointwise and in $L_{loc}^{1}$ topology.
            \item[(2)]
                $\nu(\varphi,z)- \frac{n}{k} \leq \nu(\varphi_{k},z) \leq \nu(\varphi,z)$ for $z \in \Omega$.
        \end{itemize}
    
        \noindent
        \textbf{Step 2. (Global approximation)}
        Let $u$ be a $\theta$-psh function. %Our first goal is to locally modify $u$ to construct a psh function. 
        %Let $\delta>0$. 
        Consider a finite coordinate covering of $X$. Using this covering, we construct a finite collection of open balls $U'_{j}$ of radius $\delta>0$ (with respect to the coordinates of the covering), which also cover $X$. Let $U'_{j} \subset U_{j}'' \subset U_{j} $ be concentric balls of radii $\delta$, $1.5 \delta$ and $2 \delta$ respectively, and let $\tau_{j} \colon U_{j} \rightarrow B(a_{j},2 \delta)$ be the coordinate map. 
        
        Let $\epsilon(\delta)$ be a continuous function in $\delta$ such that $\lim_{\delta \rightarrow 0} \epsilon(\delta)=0$ and $\theta_{x}-\theta_{x'} \leq \epsilon(\delta) \omega_{x}$ for $x,x' \in U_{j}$. Let $\gamma_{j}$ be the $(1,1)$-form on $B(a_{j},2 \delta)$ with constant coefficient such that $\tau_{j}^{*}(\gamma_{j})$ coincides with $-\theta- \epsilon(\delta)\omega$ at $\tau_{j}^{-1}(a_{j})$. Let $q_{j}$ be the quadratic function on $B(a_{j},2\delta)$ such that $\ddc q_{j} = \gamma_{j}$ and $u_{j} = u \circ \tau_{j}^{-1}-q_{j} \circ \tau_{j}$ on $B(a_{j},2 \delta)$. One can check that $u_{j}$ is a psh function. By \textbf{step 1.}, we get that $u_{j}$ can be approximated by 
        \[
            u_{j,k} = \frac{1}{2k} \log \sum |\sigma_{j,k}|^{2}, \quad \text{$\sigma_{j,k}$ is an orthonormal basis of $\mathscr{H}_{U_{j}}(k u_{j})$}.
        \]
        Using the gluing technique from \cite[Lemma 14.11]{Demailly_analyticmethod}, the function $u_{j,k}+ q_{j} \circ \tau_{j}$ can be glued together to form the function 
        \begin{equation}
        \label{gluingequation}
            u^{D}_{k} = \log \sum_{j} \theta_{j}^{2} \exp (2m(u_{j,k}+ q_{j} \circ \tau_{j} + \frac{C_{1}}{m}+ C_{3} \epsilon(\delta)(\delta^{2}/2-|\tau_{j}|^{2}))),
        \end{equation}
        where $0 \leq \theta_{j} \leq 1$ is the cut off function supported on $U''_{j}$ and $\theta_{j}=1$ on $U'_{j}$. We refer readers to \cite[p. 138--142 ]{Demailly_analyticmethod} for further details.

        \begin{remark}
        \label{Bergmanremark}
            The sum $\Sigma |\sigma_{\ell}|$  defined in \textbf{step 1.} is called the Bergman kernel function associated with $\mathscr{H}_{\Omega}(k \varphi)$, which can also be expressed as 
            \[
                \sup_{f \in B(1)} |f(z)|^{2},
            \]
            where $B(1)$ is the unit ball of $\mathscr{H}_{\Omega}(k \varphi)$.
        \end{remark}

        The approximation theorem preserves the ordering of singularity types, as stated below.

        \begin{lemma}
        \label{demaillycompare}
            Let $\{\theta\}$ be a pseudoeffective class, and let $u,v \in \PSH(X,\theta)$ such that $[v] \preceq [u]$. Let $u_{k}^{D}, v_{k}^{D} \in \PSH(X,\theta+ \epsilon_{k} \omega)$ be the approximation sequences of $u$ and $v$ given by Theorem \ref{demaillyapprotheorem}. Then, we have  $[v^{D}_{k}] \preceq [u_{k}^{D}]$ for each $k$.
        \end{lemma}

        \begin{proof}
            Without loss of generality, we may assume that $v \leq u$. Note that the local potentials $u_{j}$ and $v_{j}$ of $u$ and $v$ on $B(a_{j},2\delta)$ (constructed as above) still satisfies $v_{j} \leq u_{j}$. By Remark \ref{Bergmanremark}, we get that $v_{j,k} \leq u_{j,k}$ for each $k$. The constants $C_{1},C_{3}$ in (\ref{gluingequation}) can be chosen to be independent of $u$ and $v$. Combining this with the fact that $v_{j,k} \leq u_{j,k}$, we get $v^{D}_{k} \leq u^{D}_{k}$.
        \end{proof}

   The Monge-{A}mp\`ere masses of the sequence $\{u_{k}^{D}\}$ has the following convergence property.

    \begin{proposition}    
    (\cite[Proposition 3.4]{Darvas_Xia-volumelinebundle})
    \label{darvasconvergencethm}
        Let $\{\theta\}$ be a pseudoeffective class, and let $u \in \PSH(X,\theta)$. Then
        \[
            \int_{X}  \big{\langle} (\theta+ \epsilon_{k} \omega)_{u^{D}_{k}}^{n} \big{\rangle} \searrow \int_{X}  \big{\langle}(\theta)_{P^{\theta}[u]_{\mathcal{I}}}^{n} \big{\rangle} \mbox{, as $k \longrightarrow \infty$}.
        \]
    \end{proposition} 
    
    The convergence property above can be generalized to the mixed setting as follows. The result is similar to \cite[Theorem 4.2]{Xiaokounkov}. We present a proof for the reader's convenience.

    \begin{lemma}
    \label{generalizeddarvasconvergencethm}
        Let $\{\theta_{1}\}, \dots , \{\theta_{n}\}$ be big classes. Let $\delta>0$ and let $u_{j} \in \PSH(X,\theta_{j})$ such that $\int_{X} \langle (\theta_{j})_{u_{j}}^{n} \rangle> \delta$ for $1 \leq j \leq n$. Then, we have 
        \[
            \int_{X}  \big{\langle} \wedge_{j=1}^{n}(\theta_{j}+ \epsilon_{k} \omega)_{u^{D}_{j,k}} \big{\rangle}    \searrow \int_{X}\big{\langle}\wedge_{j=1}^{n}(\theta_{j})_{P^{\theta_{j}}[u_{j}]_{\mathcal{I}}} \big{\rangle}\text{, as $k \longrightarrow \infty$},
        \]
        where $u^{D}_{j,k}$ is the approximation sequence of $u_{j}$ given by Theorem \ref{demaillyapprotheorem}.
    \end{lemma}
    \begin{proof}
        Since $u_{j,k}^{D}$ is of analytic singularities, we get $[u_{j,k}^{D}]=[P^{\theta_{j}+\epsilon_{k}\omega}[u_{j,k}^{D}]_{\mathcal{I}}]$ (\cite[Theorem 4.3]{Kim}). By \cite[Proposition 3.3]{Darvas-Xia}, we have $P^{\theta_{j}+ \epsilon_{k}\omega}[u_{j,k}^{D}]_{\mathcal{I}} \searrow P^{\theta_{j}}[u_{j}]_{\mathcal{I}}$ as $k \longrightarrow \infty$. Hence,
        \begin{equation*}
            \lim_{k \rightarrow \infty} \int_{X} \langle \wedge_{j=1}^{n}(\theta_{j}+ \epsilon_{k} \omega)_{u^{D}_{j,k}} \rangle= \lim_{k \rightarrow \infty} \int_{X} \langle \wedge_{j=1}^{n}(\theta_{j}+ \epsilon_{k} \omega)_{P^{\theta_{j}+\epsilon_{k}\omega}[u^{D}_{j,k}]_{\mathcal{I}}} \rangle \geq \int_{X} \langle \wedge_{j=1}^{n}(\theta_{j})_{P^{\theta_{j}}[u_{j}]_{\mathcal{I}}} \rangle.
        \end{equation*}
        
        We now prove 
        \begin{equation}
        \label{generalizeddarvasconvergencethmeq}
            \lim_{k \rightarrow \infty} \int_{X} \langle \wedge_{j=1}^{n}(\theta_{j}+ \epsilon_{k} \omega)_{u^{D}_{j,k}} \rangle \leq \int_{X} \langle \wedge_{j=1}^{n}(\theta_{j})_{P^{\theta_{j}}[u_{j}]_{\mathcal{I}}} \rangle.
        \end{equation}
        By \cite[Lemma 4.3]{Darvas-Lu-DiNezza-singularity-metric} and Proposition \ref{darvasconvergencethm}, we can choose
        \[
            b_{j,k} \in \Big{(}1, \Big{(}\frac{\int_{X}\langle(\theta_{j}+\epsilon_{k}\omega)^{n}_{u_{j,k}^{D}}\rangle}{\int_{X}\langle(\theta_{j}+\epsilon_{k}\omega)^{n}_{u_{j,k}^{D}}\rangle-\int_{X}\langle(\theta_{j}+\epsilon_{k}\omega)^{n}_{P^{\theta_{j}}[u_{j}]_{\mathcal{I}}}\rangle}\Big{)}^{\frac{1}{n}}\Big{)},
        \]
        which increases to $\infty$ as $k \longrightarrow \infty$, and $\varphi_{j,k} \in \PSH(X,\theta_{j}+ \epsilon_{k}\omega)$ such that 
        \[
            b_{j,k}^{-1}\varphi_{j,k} + (1-b_{j,k}^{-1})u_{j,k}^{D} \leq P^{\theta_{j}}[u_{j}]_{\mathcal{I}}.
        \]
        By Theorem \ref{monotonicity}, we then have 
        \[
            \prod_{j=1}^{n}(1-b_{j,k}^{-1}) \int_{X} \langle \wedge_{j=1}^{n}(\theta_{j}+ \epsilon_{k} \omega)_{u_{j,k}^{D}} \rangle \leq \int_{X} \langle \wedge_{j=1}^{n} (\theta_{j} + \epsilon_{k} \omega )_{P^{\theta_{j}}[u_{j}]_{\mathcal{I}}} \rangle.       
        \]
        Let $k \longrightarrow \infty$, then we get (\ref{generalizeddarvasconvergencethmeq}). This completes the proof.
    \end{proof}

\section{Proof of the main theorems}

\subsection{Proof of Theorem \ref{maintheorem}}
        \textbf{Step 1.}
        We first focus on the case where $T_{j}$ has analytic singularities and there exists a K\"ahler current $P_{j} \in \{\theta_{j}\}$ with analytic singularities that is more singular than $T_{j}$ for each $j$. We also assume that $V$ is a smooth submanifold of dimension greater than $n-m$.
      %Since $T_{j}$ is big, there exists a K\"ahler current $P_{j} \in \{\theta_{j}\}$ that is more singular than $T_{j}$ (\cite[Proposition 3.6]{Darvas_Xia-volumelinebundle}).
      
      Let $\epsilon >0$ be such that $P_{j} \geq \epsilon \omega$ for $1 \leq j \leq m$. Set $\widetilde{T}_{j} = P(T_{j} - T_{j}^{p})$, where we recall that $T_{j}^{p} \defeq \mathbbm{1}_{I_{T_{j}}}T_{j}$. Our first step is to reduce the problem to the case where $T_{j}^{p}=0$ for $1\leq j \leq m$. To achieve this, we replace $T_{j}$, $T'_{j}$ and $P_{j}$ with  $\widetilde{T}_{j}$, $T'_{j}-T_{j}^{p}$ and $P_{j}-T_{j}^{p}$, respectively. By Lemma \ref{replacementlemma2}, we obtain 
       \begin{align}
       \label{replacementmassequation}
           & \| \langle \wedge_{j=1}^{m}T_{j} \rangle \|=\| \langle \wedge_{j=1}^{m} \widetilde{T}_{j}\rangle\|,\notag\\
           &\| \langle \wedge_{j=1}^{m} T'_{j} \rangle\|=\| \langle \wedge_{j=1}^{m} (T'_{j} - T_{j}^{p}) \rangle\|,\\
           &\| \langle  \wedge_{j=1}^{m} P_{j} \rangle \|=\| \langle \wedge_{j=1}^{m} (P_{j} - T_{j}^{p}) \rangle\|. \notag
       \end{align}
       %Note that $T_{j}- T^{p}_{j}$ is not necessary less singular than $T'_{j} - T^{p}_{j}$ and $P_{j}- T^{p}_{j}$. To preserve the monotonicity of the singularities types, we need to consider the difference envelope of $T_{j}$ and $T^{p}_{j}$, $\widetilde{T}_{j}$. 
       %Since $T_{j}^{p}$ is supported in an analytic set, we get $P_{j} - T^{p}_{j} \geq \epsilon \omega$ on X, hence, it is also a K\"ahler current.
       Note that the new currents $\widetilde{T}_{j}$, $T'_{j}-T_{j}^{p}$ and $P_{j} - T_{j}^{p}$ satisfy the following properties.

       \begin{proposition}
       \label{replacementproposition}
       For $1\leq j \leq m$, the following hold.
            \begin{itemize}
                \item[(1)]
                    $\widetilde{T}_{j}$ is less singular than $T'_{j}-T_{j}^{p}$ and $P_{j}-T_{j}^{p}$.
                \item[(2)] 
                    $\nu(\widetilde{T}_{j}+ T_{j}^{p} ,x) = \nu (T_{j},x)$.
                \item[(3)]                      
                    $\widetilde{T}_{j}^{p} \defeq \mathbbm{1}_{I_{\widetilde{T}_{j}}}\widetilde{T}_{j}=0$.
            \end{itemize}      
       \end{proposition}
  
       \begin{proof}

            The proof of (1) follows directly from the definition of difference envelope of currents (Definition \ref{differenceofcurrents}). As for (2), we first write $T_{j}$, $T_{j}^{p}$, and $\widetilde{T}_{j}$ as 
                    \begin{align*}
                        &T_{j} = \ddc u_{j} + \theta_{j}, \\
                        &T^{p}_{j} = \ddc u^{p}_{j} + \theta_{j}^{p}, \\
                        & \widetilde{T}_{j} = \ddc \widetilde{u}_{j} + (\theta_{j} - \theta^{p}_{j}),
                    \end{align*}
                    where $\widetilde{u}_{j} = (\sup \{ v \in \PSH(X , \theta_{j} - \theta_{j}^{p}) | v \leq u_{j}-u_{j}^{p}\})^{*}$. Since $u_{j}$ has analytic singularities, the singularity types of $u_{j}$ and $P^{\theta_{j}}[u_{j}]$ are the same (see \cite[Lemma 3.2]{Darvas_Xia-volumelinebundle}), and we obtain
                    \begin{equation*}
                    %\label{replacelelongeq}
                        [u_{j}] = [P^{\theta_{j}}[u_{j}]] \succeq [\widetilde{u}_{j} + u_{j}^{p}],
                    \end{equation*}
                    On the other hands, since 
                    \[
                        T_{j}-T_{j}^{p} \preceq P(T_{j}-T_{j}^{p}) = \widetilde{T}_{j},
                    \]
                    we deduce that $[\widetilde{u}_{j} + u_{j}^{p}] \succeq [u_{j}]$. 
                    %Combining this with (\ref{replacelelongeq}), 
                    Hence, we get $[\widetilde{u}_{j}+u_{j}^{p}]=[u_{j}]$ and
                    \[
                        \nu(\widetilde{T}_{j}+T_{j}^{p},x) = \nu (T_{j},x),
                    \]
                    for $x \in X$.

                    Now, let's move on to (3). Note that $ P_{j}=P_{j}- T_{j}^{p} \geq \epsilon \omega$ on $X \backslash V$, where $V = I_{T_{j}}$ is a hypersurface on $X$. Since $\omega$ is smooth, the above inequality extends to entire $X$. Therefore, we get that $P_{j}-T_{j}^{p}$ is a K\"ahler current with analytic singularities. By applying Lemma \ref{replacementlemma2}, we get 
                    \[
                        0= \mathbbm{1}_{I_{T_{j}-T_{j}^{p}}}(T_{j}-T_{j}^{p}) \geq \mathbbm{1}_{I_{\widetilde{T}_{j}}} \widetilde{T}_{j}.
                    \]
                    Thus, $\mathbbm{1}_{I_{\widetilde{T}_{j}}}\widetilde{T}_{j}=0$.
            %\begin{itemize}
            %    \item[(1)] 
            %        First, we write $T_{j}$, $T^{p}_{j}$, $P(T_{j}- T^{p}_{j})$ as 
            %        \begin{align*}
            %            &T_{j} = \ddc u_{j} + \theta_{j}, \\
            %            &T^{p}_{j} = \ddc u^{p}_{j} + \theta^{p}_{j}, \\
            %            & P(T_{j} - T^{p}_{j}) = \ddc u^{e}_{j} + (\theta_{j} - \theta^{p}_{j}),
            %        \end{align*}
            %        where $u^{e}_{j} = (\sup \{ v \in \PSH(X , \theta_{j} - \theta^{p}_{j}) | v \leq 0, [v+ u^{p}_{j}] \preceq [u_{j}]\})^{*}$. Since $u_{j}$ has analytic singularities, we get 
            %        \begin{equation}
            %        \label{replacelelongeq}
            %            [u_{j}] = [P^{\theta_{j}}[u_{j}]] \succeq [u_{j}^{e} + u^{p}_{j}],
            %        \end{equation}
            %        where the first equality comes from \cite[Lemma 3.2]{Darvas_Xia-volumelinebundle}. On the other hands, since 
            %        \[
            %            [T_{j}-T^{p}_{j}] \preceq [P(T_{j}-T^{p}_{j})],
            %        \]
            %        we then have $[u^{e}_{j} + u^{p}_{j}] \succeq [u_{j}]$. This combined with (\ref{replacelelongeq}) shows 
            %        \[
            %            \nu(P(T_{j}-T^{p}_{j})+T^{p}_{j},x) = \nu (T_{j},x),
            %        \]
            %        for any $x \in X$.
            %    \item[(2)]
            %        By Lemma \ref{replacementlemma2}, we only need to show that there exists a K\"ahler current with analytic ssingularitie in $\{\theta_{j} - \theta^{p}_{j}\}$ that is more singular than $\widetilde{T}_{j}$
            %\end{itemize}
       \end{proof}
       The equalities (\ref{replacementmassequation}) and  Proposition \ref{replacementproposition} explain why it suffices to prove the theorem with the currents $\widetilde{T}_{j}$, $T'_{j}-T_{j}^{p}$ and $P_{j}-T_{j}^{p}$. For convenience, from now on, we will continuously use $T_{j}$, $T'_{j}$ and $P_{j}$ to represent the currents $\widetilde{T}_{j}$, $T'_{j}-T_{j}^{p}$ and $P_{j}-T_{j}^{p}$, respectively.

      Let $\sigma \colon \widehat{X} \longrightarrow X$ be the blow-up of $X$ along $V$. We denote by $ \widehat{V} = \sigma^{-1}(V)$ the exceptional divisor. In general, the pull back of the K\"ahler form $\omega$ by $\sigma$ is not necessarily a K\"ahler form. However, we can construct a K\"ahler form on $\widehat{X}$ as following. By \cite[Lemma 3.25]{Voisin1}, there exists a closed smooth $(1,1)$-form $\omega_{h}$ that is cohomologous to $-[\widehat{V}]$, and a constant $c_{V}>0$ depends on $V$ such that 
       \begin{equation}
       \label{blowupkahler}
            \widehat{\omega} \defeq c_{V} \sigma^{*} \omega + \omega_{h}>0,
       \end{equation}
       which is a K\"ahler form on $\widehat{X}$. 
       
       For $\delta \in (0,1)$, set $P_{j}^{\delta} \defeq (1-\delta) T_{j} + \delta P_{j}$. Note that $P^{\delta}_{j} \geq \delta \epsilon \omega$ is a K\"ahler current on $X$. Now, we decompose the pull back of $T_{j}$, $T'_{j}$ and $P^{\delta}_{j}$ by $\sigma$ as follows
       \begin{align*}
           &\sigma^{*}T_{j} = \lambda_{j} [ \widehat{V}]+ \eta_{j},\\
           &\sigma^{*}T'_{j} = \lambda'_{j} [\widehat{V}] + \eta'_{j},\\
           &\sigma^{*}P_{j}^{\delta} = \lambda_{j}^{\delta}[\widehat{V}] + \eta_{j}^{\delta}.
       \end{align*}
       Since the Lelong numbers preserved by the blow-up map (\cite[Corollary 1.1.8]{Boucksom-these}), we get 
       \begin{align*}
           \lambda_{j} &= \nu(\sigma^{*}T_{j},\widehat{V})= \nu(T_{j},V),\\
           \lambda'_{j} &= \nu(\sigma^{*}T'_{j},\widehat{V})= \nu(T'_{j},V),\\
           \lambda_{j}^{\delta} &= \nu(\sigma^{*}P_{j}^{\delta}, \widehat{V})= \nu(P_{j}^{\delta},V).
       \end{align*}

       We note that for any closed $(n-m,n-m)$-form $\Phi$, we have
       \[
            \int_{X} \langle \wedge_{j=1}^{m} T_{j}  \rangle \wedge \Phi = \int_{\widehat{X}} \langle \wedge_{j=1}^{m}\eta_{j} \rangle \wedge \sigma^{*} \Phi \mbox{ and } \int_{X} \langle \wedge_{j=1}^{m} T'_{j} \rangle \wedge \Phi = \int_{\widehat{X}} \langle \wedge_{j=1}^{m} \eta'_{j} \rangle \wedge \sigma^{*}\Phi
       \]
       
       \begin{lemma}
       \label{nonpluriclassicallemma}
           Let $\eta_{m}$ be the $(1,1)$-current defined as above. Let $\tau_{1}, \dots, \tau_{m-1}$ be closed positive $(1,1)$-currents on $\widehat{X}$. Then, we have
           \[
                \langle \wedge_{j=1}^{m-1} \tau_{j} \wedge \eta_{m} \rangle = \langle \wedge_{j=1}^{m-1} \tau_{j}  \dot{\wedge} \eta_{m} \rangle.
           \]
           
       \end{lemma}

       \begin{proof}
           By \cite[Proposition 3.6]{Viet-generalized-nonpluri}, it suffices to show that $\mathbbm{1}_{I_{\eta_{m}}}\eta_{m}=0$. Note that 
           \begin{align*}
               \mathbbm{1}_{I_{\eta_{m}}} \eta_{m} &= \mathbbm{1}_{I_{\eta_{m}} \backslash \widehat{V}} \eta_{m}\\
               &\leq \sigma^{*} (\mathbbm{1}_{\sigma (I_{\eta_{m}})} T_{m})\\
               &\leq \sigma^{*} (\mathbbm{1}_{I_{T_{m}}} T_{m}).
           \end{align*}
           Since we assume $T_{m}^{p}=\mathbbm{1}_{I_{T_{m}}}T_{m}=0$, it follows that $\mathbbm{1}_{I_{\eta_{m}}}\eta_{m}=0$, which completes the proof.
       \end{proof}

       Note that $T_{j}$ is less singular than $P_{j}^{\delta}$, hence
       \[
           \sigma^{*}P_{j}^{\delta}=\eta_{j}^{\delta} + \lambda^{\delta}_{j}[\widehat{V}] \preceq  %\eta_{j} + \lambda_{j}[\widehat{V}]= 
           \sigma^{*}T_{j}.
       \]
       This implies 
       \begin{equation}
       \label{replaceinequality}
           \eta_{j}^{\delta}  \preceq P((\sigma^{*}T_{j})-\lambda^{\delta}_{j}[\widehat{V}]).
       \end{equation}
       Set 
       \[
            Q_{j}^{\delta} \defeq \eta_{j}^{\delta}+ \frac{\delta\epsilon}{ 2 c_{V}} \omega_{h} \geq \frac{\delta \epsilon}{2 c_{V}} \widehat{\omega}, \quad \widetilde{\eta}_{j}^{\delta} \defeq P(P(\sigma^{*}T_{j}- \lambda^{\delta}_{j}[\widehat{V}]) -\frac{\delta \epsilon}{2 c_{V}}[\widehat{V}]).
       \]

       \begin{lemma}
       \label{qcomparelemma}
            The currents $ Q^{\delta}_{j}$ is more singular than $\widetilde{\eta}^{\delta}_{j}$.
       \end{lemma}

       \begin{proof}
           First, we write
           \begin{align*}
               \eta_{j} &= \ddc u_{\eta_{j}} + \sigma^{*}\theta_{j}+ \lambda_{j}\omega_{h},& u_{\eta_{j}}&\in \PSH(\widehat{X},\sigma^{*}\theta_{j}+\lambda_{j}\omega_{h}),\\
               \eta_{j}^{\delta}&=\ddc u_{j}^{\delta}+ \sigma^{*}\theta_{j} + \lambda_{j}^{\delta} \omega_{h},& u_{j}^{\delta} &\in \PSH(\widehat{X},\sigma^{*}\theta_{j} + \lambda_{j}^{\delta} \omega_{h}),\\
               P(\sigma^{*}T_{j}- \lambda_{j}^{\delta}[\widehat{V}])&= \ddc u_{j}^{p} + \sigma^{*}\theta_{j} + \lambda_{j}^{\delta} \omega_{h},&  u_{j}^{p} &\in \PSH(\widehat{X}, \sigma^{*}\theta_{j} + \lambda_{j}^{\delta} \omega_{h}),\\
               [\widehat{V}]&=\ddc u_{\widehat{V}} - \omega_{h} ,& u_{\widehat{V}} &\in \PSH(\widehat{X}, - \omega_{h}).
           \end{align*}
           We normalize the potential function $u_{\widehat{V}}$ so that $u_{\widehat{V}} \leq 0$. Note that $u_{j}^{\delta}$ can also represent the potential of $Q_{j}^{\delta}$. In other words, $u_{j}^{\delta} \in \PSH(X,\sigma^{*}\theta_{j} + (\lambda_{j}^{\delta} +\frac{\delta\epsilon}{2 c_{V}})\omega_{h})$ such that 
           \[
                Q_{j}^{\delta} = \ddc u_{j}^{\delta} + \sigma^{*}\theta_{j} + (\lambda_{j}^{\delta}  +\frac{\delta\epsilon}{2 c_{V}})\omega_{h}.
           \]
           By (\ref{replaceinequality}) and the fact that $u_{\widehat{V}} \leq 0$, we obtain
            \[
                [u_{j}^{\delta} +  \frac{\delta\epsilon}{2c_{V}} u_{\widehat{V}}] \preceq [u_{j}^{\delta}]  \preceq [u_{j}^{p}],
           \] 
           which implies $Q_{j}^{\delta} \preceq P(P(\sigma^{*}T_{j}- \lambda^{\delta}_{j}[\widehat{V}]) -\frac{\delta \epsilon}{2 c_{V}}[\widehat{V}])=\widetilde{\eta}_{j}^{\delta}$.
       \end{proof}

       We prove Theorem \ref{maintheorem} by contradiction. Suppose $\lambda'_{j}=\nu(T'_{j},V) > \nu(T_{j},V) = \lambda_{j}$ for $1 \leq j \leq m$. For each $j$, take $\delta_{j}>0$ small enough such that $\lambda'_{j} - \lambda^{\delta_{j}}_{j} -\frac{\delta_{j}\epsilon}{2 c_{V}}>0$. Here, we provide a precise method for choosing $\delta_{j}$. First, note that 
       \begin{align*}
           \lambda'_{j}-\lambda^{\delta_{j}}_{j}-\frac{\delta_{j}\epsilon}{2c_{V}}=(\lambda'_{j}-\lambda_{j})- \delta_{j}\big{(}(\lambda^{P_{j}}-\lambda_{j})-\frac{\epsilon}{2c_{V}} \big{)},
       \end{align*}
       where $\lambda^{P_{j}} \defeq \nu(\sigma^{*}P_{j},\widehat{V})=\nu(P_{j},V)$. Since $P_{j}$ is a current in $\{\theta_{j}\}$, we get $\lambda^{P_{j}} \leq c\|\{\theta_{j}\}\|$ for some constant $c$ independent of the class $\{\theta_{j}\}$. Combining this with the equality above, we can  choose 
       \begin{equation}
       \label{deltachoosing}
            \delta_{j}=(c \|\{\theta_{j}\}\|+\frac{\epsilon}{2c_{V}})^{-1}\frac{\lambda'_{j}-\lambda_{j}}{2}.
       \end{equation}
       %\begin{equation}
       %\label{deltachoosing}
       %     \delta=(c_{2}+\frac{\epsilon}{2c})^{-1}\frac{\widetilde{\lambda}}{2}, \quad \mbox{where $c_{2} = c_{1}\max\{\|\{\theta_{j}\}\|\}$ and $\widetilde{\lambda} = \min\{\lambda'_{j} - \lambda_{j}\}$}.
       %\end{equation}
     
       \begin{lemma}
       \label{compareremark}
           $\sigma^{*}T'_{j}-(\lambda^{\delta_{j}}_{j}+\frac{\delta_{j}\epsilon}{2c_{V}})[\widehat{V}]$ is a closed positive current and is more singular than $\widetilde{\eta}^{\delta_{j}}_{j}$.
       \end{lemma}

       \begin{proof}
           Since $T_{j}$ is less singular than $T'_{j}$, we get 
           \[
                \sigma^{*}T'_{j} - \lambda^{\delta_{j}}_{j}[\widehat{V}] \preceq P(\sigma^{*}T_{j} - \lambda^{\delta_{j}}_{j}[\widehat{V}]).
           \]
           Note that $ \sigma^{*}T'_{j} - \lambda^{\delta_{j}}_{j}[\widehat{V}]= \eta'_{j} + (\lambda'_{j}- \lambda^{\delta_{j}}_{j})[\widehat{V}]$ is a closed positive current. This is because $\lambda'_{j}-\lambda^{\delta_{j}}_{j}>\lambda'_{j}-\lambda^{\delta_{j}}_{j}-\frac{\delta_{j}\epsilon}{2c_{V}}>0$. By the similar process as above, we obtain that  $\sigma^{*}T'_{j}-(\lambda^{\delta_{j}}_{j}+\frac{\delta_{j}\epsilon}{2c_{V}})[\widehat{V}] = \eta'_{j}+ (\lambda'_{j} - \lambda^{\delta_{j}}_{j}-\frac{\delta_{j}\epsilon}{2c_{V}})[\widehat{V}]$ is a closed positive current, and satisfied 
           \[
                \sigma^{*}T'_{j}-(\lambda^{\delta_{j}}_{j}+\frac{\delta_{j}\epsilon}{2c})[\widehat{V}] \preceq P(P(\sigma^{*}T_{j}- \lambda^{\delta_{j}}_{j}[\widehat{V}]) -\frac{\delta_{j} \epsilon}{2 c_{V}}[\widehat{V}]) = \widetilde{\eta}^{\delta_{j}}_{j}.
           \]
       \end{proof}

       \begin{lemma}
       \label{mainlemma}
       %For any closed positive smooth $(n-m,n-m)$-form $\Phi$ on $X$, we have
       %\begin{align*}
       %    \int_{\widehat{X}} \langle\wedge_{j=1}^{m} \eta'_{j} \rangle \wedge \sigma^{*} \Phi \leq \int_{\widehat{X}} \langle \wedge_{j=1}^{m} \eta_{j} \rangle \wedge \sigma^{*} \Phi - (\lambda'_{m}-\lambda_{m})\int_{\widehat{X}} \langle \wedge_{j=1}^{m-1} \widetilde{\eta}^{\delta}_{j} \dot{\wedge} [\widehat{V}] \rangle \wedge \sigma^{*}\Phi.
       %\end{align*}
       We have the following inequality.
       \[
            \{ \langle\wedge_{j=1}^{m} \eta'_{j} \rangle \} +\{(\lambda'_{m}-\lambda_{m})\langle \wedge_{j=1}^{m-1} \widetilde{\eta}^{\delta_{j}}_{j} \dot{\wedge} [\widehat{V}] \rangle \} \leq \{\langle\wedge_{j=1}^{m}P[\sigma^{*}T_{j}] \rangle \}.
       \]

       %$\|\langle \wedge_{j=1}^{n} \eta_{j}  \rangle\| \leq \|\langle \eta^{n} \rangle \| - \| \langle \wedge_{j=1}^{n-1} \widetilde{\eta}^{\delta_{j}}\dot{\wedge} (\lambda_{n} - \lambda)[\widehat{V}] \rangle \|.$
       %\mbox{}
       %\begin{itemize}
       %    \item[(1)]
       %         $\|\langle \wedge_{j=1}^{n} \eta_{j}  \rangle\| \leq \|\langle \wedge_{j=1}^{n-1}\widetilde{\eta}^{\delta_{j}} \wedge \eta \rangle \| - \| \langle \wedge_{j=1}^{n-1} \widetilde{\eta}^{\delta_{j}}\dot{\wedge} (\lambda_{n} - \lambda)[\widehat{V}] \rangle \|.$
       %    \item[(2)]
       %         $\|\langle \wedge_{j=1}^{n-1} \widetilde{\eta}^{\delta_{j}} \wedge \eta \rangle\| \leq \|\langle \eta^{n}  \rangle\|$.
       %\end{itemize}
       \end{lemma}

       \begin{proof}
            By the multi-linearity and the monotonicity of non-pluripolar product (\cite[Proposition 1.4]{BEGZ} and Theorem \ref{monotonicity}),
            %\cite[Proposition 1.4]{BEGZ}), 
            we obtain
            \begin{align*}
                 &\{ \langle\wedge_{j=1}^{m} \eta'_{j} \rangle \} +\{(\lambda'_{m}-\lambda_{m})\langle \wedge_{j=1}^{m-1} \widetilde{\eta}^{\delta_{j}}_{j} \dot{\wedge} [\widehat{V}] \rangle \}\\
                 \leq &\{\langle \wedge_{j=1}^{m-1} (\eta'_{j}+(\lambda'_{j}-\lambda^{\delta_{j}}_{j}-\frac{\delta_{j}\epsilon}{2c_{V}})[\widehat{V}])
                \dot{\wedge} \eta'_{m}\rangle\}+\{(\lambda'_{m}-\lambda_{m})\langle \wedge_{j=1}^{m-1} \widetilde{\eta}^{\delta_{j}}_{j} \dot{\wedge} [\widehat{V}] \rangle \}\\
                \leq &\{\langle \wedge_{j=1}^{m-1} \widetilde{\eta}^{\delta_{j}}_{j} \dot{\wedge} \eta'_{m} \rangle\}+\{(\lambda'_{m}-\lambda_{m})\langle \wedge_{j=1}^{m-1} \widetilde{\eta}^{\delta_{j}}_{j} \dot{\wedge} [\widehat{V}] \rangle \}\mbox{ (by Lemma \ref{compareremark})}\\
                = &\{\langle \wedge_{j=1}^{m-1} \widetilde{\eta}^{\delta_{j}}_{j} \dot{\wedge} (\eta'_{m}+ (\lambda'_{m} - \lambda_{m})[\widehat{V}]) \rangle\}\\
                \leq &\{\langle \wedge_{j=1}^{m-1} \widetilde{\eta}^{\delta_{j}}_{j} \dot{\wedge} \eta_{m} \rangle\} \\ 
                = &\{\langle \wedge_{j=1}^{m-1} \widetilde{\eta}^{\delta_{j}}_{j} \wedge \eta_{m} \rangle \}\mbox{ (by Lemma \ref{nonpluriclassicallemma})}
            \end{align*}

            %\begin{align*}
            %    &\|\langle \wedge_{j=1}^{n} \eta_{j}\rangle\|+ \|\langle \wedge_{j=1}^{n-1}\widetilde{\eta}^{\delta_{j}}  \dot{\wedge} (\lambda_{n} - \lambda) [\widehat{V}] \rangle\|\\ 
            %    \leq &\|\langle \wedge_{j=1}^{n-1} (\eta_{j}+(\lambda_{j}-\lambda^{\delta_{j}}-\frac{\delta_{j}\epsilon}{2c})[\widehat{V}])
            %    \dot{\wedge} \eta_{n}\rangle\|+\|\langle \wedge_{j=1}^{n-1} \widetilde{\eta}^{\delta_{j}} \dot{\wedge} (\lambda_{n} - \lambda) [\widehat{V}] \rangle\|\\
            %    \leq &\|\langle \wedge_{j=1}^{n-1} \widetilde{\eta}^{\delta_{j}}\dot{\wedge} \eta_{n}\rangle\|+\|\langle \wedge_{j=1}^{n-1} \widetilde{\eta}^{\delta_{j}} \dot{\wedge} (\lambda_{n} - \lambda) [\widehat{V}] \rangle\| \mbox{ (by Lemma \ref{compareremark})}\\
            %    = &\|\langle  \wedge_{j=1}^{n-1} \widetilde{\eta}^{\delta_{j}}  \dot{\wedge} (\eta_{n} + (\lambda_{n} - \lambda)[\widehat{V}]) \rangle\|\\
            %    \leq &\|\langle \wedge_{j=1}^{n-1} \widetilde{\eta}^{\delta_{j}} \dot{\wedge} \eta \rangle\|\\
            %    = &\|\langle \wedge_{j=1}^{n-1} \widetilde{\eta}^{\delta_{j}} \wedge \eta \rangle\| \mbox{ (by Lemma \ref{nonpluriclassicallemma})}
            %\end{align*}
            Now, we show that $\{\langle \wedge_{j=1}^{m-1} \widetilde{\eta}^{\delta_{j}}_{j} \wedge \eta_{m} \rangle\} \leq \{\langle\wedge_{j=1}^{m}P[\sigma^{*}T_{j}] \rangle \}$. Let 
            \[                              \widetilde{\theta}^{\delta_{j}}_{j}=\sigma^{*}\theta_{j} + (\lambda_{j}^{\delta_{j}} +\frac{\delta_{j}\epsilon}{2 c_{V}})\omega_{h}.%\theta_{\eta_{j}}-(\lambda_{j} - \lambda_{j}^{\delta_{j}}- \frac{\delta_{j}\epsilon}{2c}) \omega_{h}
            \]
            We write $\widetilde{\eta}^{\delta_{j}}_{j}=\ddc \widetilde{u}^{\delta_{j}}_{j}+\widetilde{\theta}^{\delta_{j}}_{j}$, $\widetilde{u}^{\delta_{j}}_{j} \in \PSH(\widehat{X},\widetilde{\theta}^{\delta_{j}}_{j})$, and recall that $P(\sigma^{*}T_{j}- \lambda^{\delta_{j}}_{j}[\widehat{V}])=\ddc u_{j}^{p} + (\widetilde{\theta}^{\delta_{j}}_{j}-\frac{\delta_{j}\epsilon}{2c_{V}}\omega_{h})$, where
            %Now, we are going to show that $\|\langle \wedge_{j=1}^{n-1} \widetilde{\eta}^{\delta_{j}} \wedge \eta \rangle\| \leq \|\langle \eta^{n} \rangle\|$. First, we express the currents $\widetilde{\eta}^{\delta_{j}}$ and $P(\sigma^{*}T- \lambda^{\delta_{j}}[\widehat{V}])$ as 
            %\[
            %    \widetilde{\eta}^{\delta_{j}}=\ddc \widetilde{u}^{\delta_{j}}+\widetilde{\theta}^{\delta_{j}} \text{ and }P(\sigma^{*}T- \lambda^{\delta_{j}}[\widehat{V}]) = \ddc u'_{j}+ \theta'_{j}
            %]
            %respectively, where $\theta'_{j} = \widetilde{\theta}^{\delta_{j}}-\frac{\delta_{j} \epsilon}{2c}\omega_{h}$ and 
            \begin{align*}
                &\widetilde{u}^{\delta_{j}}_{j} = \big{(}\sup\{v \in \PSH(\widehat{X},\widetilde{\theta}^{\delta_{j}}_{j}) |v + \frac{\delta_{j}\epsilon}{2c_{V}}u_{\widehat{V}} \leq u^{p}_{j}\}\big{)}^{*},\\
                &u^{p}_{j} = (\sup\{v \in \PSH(\widehat{X},\widetilde{\theta}^{\delta_{j}}_{j}-\frac{\delta_{j}\epsilon}{2c_{V}}\omega_{h}) | v+\lambda^{\delta_{j}}_{j}u_{\widehat{V}} \leq \sigma^{*}u_{j}\})^{*}.
            \end{align*}
            One observes that $[\widetilde{u}^{\delta_{j}}_{j}+ \frac{\delta_{j} \epsilon}{ 2 c_{V}} u_{\widehat{V}}] \preceq [P[u^{p}_{j}]]$ and $[u^{p}_{j}+ \lambda^{\delta_{j}}_{j}u_{\widehat{V}}] \preceq [P[\sigma^{*}u_{j}]]$. This implies 
            \begin{equation*}
            \label{comparelemmaineq}
                [\widetilde{u}^{\delta_{j}}_{j} + (\frac{\delta_{j} \epsilon}{2c_{V}}+ \lambda^{\delta_{j}}_{j})u_{\widehat{V}}] \preceq [P[u^{p}_{j}]+ \lambda^{\delta_{j}}_{j}u_{\widehat{V}}] \preceq [P[u^{p}_{j}+ \lambda^{\delta_{j}}_{j} u_{\widehat{V}}]] \preceq [P[P[\sigma^{*}u_{j}]]]=[P[\sigma^{*}u_{j}]].
            \end{equation*}
            In other words, $\widetilde{\eta}^{\delta_{j}}_{j}+(\frac{\delta_{j}\epsilon}{2c_{V}}+\lambda^{\delta_{j}}_{j})[\widehat{V}] \preceq P[\sigma^{*}T_{j}]$.
            By the multi-linearity and the monotonicity of non-pluripolar product again, we get 
            \begin{align*}
                \{\langle \wedge_{j=1}^{m-1} \widetilde{\eta}^{\delta_{j}}_{j} \wedge \eta_{m} \rangle \} &\leq \{\langle \wedge_{j=1}^{m-1}(\widetilde{\eta}^{\delta_{j}}_{j}+(\frac{\delta_{j}\epsilon}{2c_{V}}+\lambda^{\delta_{j}}_{j})[\widehat{V}]) \wedge \eta_{m} \rangle\}\\
                &\leq\{\langle\wedge_{j=1}^{m}P[\sigma^{*}T_{j}] \rangle \}
            \end{align*}

            %the following inequalities hold.
            %\begin{align*}
            %    \|\langle \wedge_{j=1}^{n-1} \widetilde{\eta}^{\delta_{j}} \wedge \eta \rangle \| &\leq \| \langle \wedge_{j=1}^{n-1} (\widetilde{\eta}^{\delta_{j}}+(\frac{\delta_{j}\epsilon}{2c}+\lambda^{\delta_{j}})[\widehat{V}]) \wedge \eta \rangle  \|\\
            %    &\leq \| \langle  P[\sigma^{*}T]^{n-1}  \wedge \eta \rangle\|\\
            %    &\leq \| \langle P[\sigma^{*}T]^{n}  \rangle\|.
            %\end{align*}
            %Finally, by \cite[Theorem 3.14]{Lu-Darvas-DiNezza-mono}, we get 
            %\[
            %    \| \langle P[\sigma^{*}T]^{n} \rangle\|= \|\langle (\sigma^{*}T)^{n} \rangle\|=\|\langle \eta^{n}  \rangle\|,
            %\]
            %and this completes the proof.
       \end{proof}
       %Let $\Phi$ be a closed positive smooth $(n-m,n-m)$-form
       By Lemma \ref{mainlemma}, Lemma \ref{qcomparelemma} and \cite[Theorem 3.14]{Lu-Darvas-DiNezza-mono}, we now have 
       \begin{align}
           &\int_{X} \langle\wedge_{j=1}^{m} T'_{j} \rangle \wedge \omega^{n-m}\notag\\ 
           =&\int_{\widehat{X}} \langle \wedge_{j=1}^{m} \eta'_{j} \rangle \wedge (\sigma^{*}\omega)^{n-m}\notag\\
           \leq &\int_{\widehat{X}} \langle \wedge_{j=1}^{m} P [\sigma^{*}T_{j}] \rangle \wedge (\sigma^{*}\omega)^{n-m}-(\lambda'_{m}-\lambda_{m}) \int_{\widehat{X}}\langle \wedge_{j=1}^{m-1} \widetilde{\eta}^{\delta_{j}}_{j} \dot{\wedge} [\widehat{V}] \rangle \wedge (\sigma^{*}\omega)^{n-m}\notag\\
           = &\int_{\widehat{X}} \langle \wedge_{j=1}^{m} \sigma^{*} T_{j} \rangle \wedge (\sigma^{*}\omega)^{n-m}-(\lambda'_{m}-\lambda_{m}) \int_{\widehat{X}}\langle \wedge_{j=1}^{m-1} \widetilde{\eta}^{\delta_{j}}_{j} \dot{\wedge} [\widehat{V}] \rangle \wedge (\sigma^{*}\omega)^{n-m}\notag\\
           %= &\int_{X} \langle \wedge_{j=1}^{m}  T_{j} \rangle \wedge \omega^{n-m}-(\lambda'_{m}-\lambda_{m}) \int_{\widehat{X}}\langle \wedge_{j=1}^{m-1} \widetilde{\eta}^{\delta_{j}}_{j} \dot{\wedge} [\widehat{V}] \rangle \wedge (\sigma^{*}\omega)^{n-m}\notag\\
           \leq&\int_{X} \langle \wedge_{j=1}^{m}  T_{j} \rangle \wedge \omega^{n-m}-(\lambda'_{m}-\lambda_{m}) \int_{\widehat{X}}\langle \wedge_{j=1}^{m-1} Q^{\delta_{j}}_{j} \dot{\wedge} [\widehat{V}] \rangle \wedge (\sigma^{*}\omega)^{n-m}\label{inequ1}
       \end{align}

        %\begin{align}
        %        \int_{X} \langle \wedge_{j=1}^{n} T_{j} \rangle &= \int_{\widehat{X}} \langle \wedge_{j=1}^{n} \eta_{j} \rangle \notag \\
        %        &\leq \int_{\widehat{X}} \langle \eta^{n} \rangle - (\lambda_{n}-\lambda)\int_{\widehat{X}} \langle \wedge_{j=1}^{n-1} \widetilde{\eta}^{\delta_{j}} \dot{\wedge} [\widehat{V}] \rangle \notag \\
        %        &= \int_{X} \langle T^{n} \rangle - (\lambda_{n}-\lambda)\int_{\widehat{X}} \langle \wedge_{j=1}^{n-1} \widetilde{\eta}^{\delta_{j}} \dot{\wedge} [\widehat{V}] \rangle.\label{inequ1}
        %    \end{align}
            %Since $\widetilde{\eta}^{\delta_{j}}$ is less singular than $Q^{\delta_{j}}$ (Lemma \ref{qcomparelemma}), we then have 
            %\begin{equation}
            %\label{inequ2}
            %    \int_{\widehat{X}} \langle  \wedge_{j=1}^{n-1}\widetilde{\eta}^{\delta_{j}} \dot{\wedge} [\widehat{V}] \rangle \geq \int_{\widehat{X}} \langle \wedge_{j=1}^{n-1} Q^{\delta_{j}}  \dot{\wedge} [\widehat{V}] \rangle.
            %\end{equation}
                
        Recall that $Q^{\delta_{j}}_{j} = \eta^{\delta_{j}}_{j} + \frac{\delta_{j} \epsilon}{ 2 c_{V}} \omega_{h} \geq \frac{\delta_{j} \epsilon}{2c_{V}} \widehat{\omega}$ and $\sigma^{*}P^{\delta_{j}}_{j} = \lambda^{\delta_{j}}_{j} [\widehat{V}]+ \eta^{\delta_{j}}_{j}$. Since $P^{\delta_{j}}_{j}$ is of analytic singularities, so do $\eta^{\delta_{j}}_{j}$ and $Q^{\delta_{j}}_{j}$. This combines with the fact that $[\widehat{V}] \not\subset I_{\eta^{\delta_{j}}_{j}}$, induces that $[\widehat{V}]$ has no mass on $I_{Q^{\delta_{j}}_{j}}$. Therefore, we can apply \cite[Proposition 3.5]{Viet-generalized-nonpluri} and get 
            \begin{align}
                \int_{\widehat{X}} \langle \wedge_{j=1}^{m-1} Q_{j}^{\delta_{j}} \dot{\wedge} [\widehat{V}] \rangle \wedge (\sigma^{*}\omega)^{n-m} &\geq \prod_{j=1}^{m-1}\big{(}\frac{\delta_{j} \epsilon}{2c_{V}}\big{)}\int_{\widehat{X}} \langle \widehat{\omega}^{m-1} \dot{\wedge} [\widehat{V}] \rangle \wedge (\sigma^{*}\omega)^{n-m} \notag\\ 
                &= \prod_{j=1}^{m-1}\big{(}\frac{\delta_{j} \epsilon}{2c_{V}}\big{)}\int_{\widehat{X}}  [\widehat{V}]  \wedge (\sigma^{*}\omega)^{n-m} \wedge \widehat{\omega}^{m-1} \neq 0\label{inequ3}
            \end{align}
            Combining (\ref{inequ1}) and (\ref{inequ3}), we get 
            \begin{equation}
            \label{inequ4}
                \int_{X} \langle \wedge_{j=1}^{m} T'_{j} \rangle \wedge \omega^{n-m} \leq \int_{X} \langle \wedge_{j=1}^{m} T_{j} \rangle \wedge \omega^{n-m} -(\lambda'_{m}- \lambda_{m})  \prod_{j=1}^{m-1}\big{(}\frac{\delta_{j} \epsilon}{2c_{V}}\big{)} \int_{\widehat{X}}  [\widehat{V}]  \wedge (\sigma^{*}\omega)^{n-m} \wedge \widehat{\omega}^{m-1}.
            \end{equation}
            %Since we assume that $T_{1}, \dots , T_{n}$ is of relative full mass intersection with respect to $T, \dots , T$, we then get a contradiction. Therefore, $\lambda_{j}-\lambda=\nu(T_{j},x_0)-\nu(T,x_{0})=0$ for some $j =1 , \dots , n$. 
            \\
            \noindent
            \textbf{Step 2.}
            Now, we remove the analytic singularities assumption of $T_{j}$ and $P_{j}$. Here, we note that the existence of the K\"ahler current $P_{j} \preceq T_{j} $ follows from \cite[Proposition 3.6]{Darvas_Xia-volumelinebundle}. We apply Demailly's approximation theorem (Theorem \ref{demaillyapprotheorem}) on the potentials of $T_{j}$, $T'_{j}$  and $P_{j}$ ($u_{j}$, $u'_{j}$ and $p_{j}$), and get sequences $u^{D}_{j,k},u'^{D}_{j,k},p^{D}_{j,k} \in \PSH(X,\theta_{j}+ \epsilon_{k} \omega)$, where $\epsilon_{k}$ decreases to $0$ such that 
        \begin{itemize}
            \item[(1)] $u_{j,k}^{D} \searrow u_{j}$, $u'^{D}_{j,k} \searrow u'_{j}$ and $p_{j,k}^{D} \searrow p_{j}$.
            \item[(2)] $u^{D}_{j,k}$, $u'^{D}_{j,k}$ and $p_{j,k}^{D}$ have analytic singularities.
            \item[(3)] $\nu(T^{D}_{j,k},x) \longrightarrow \nu(T_{j},x)$, $\nu(T'^{D}_{j,k},x) \longrightarrow \nu(T'_{j},x)$ and $\nu(P^{D}_{j,k},x) \longrightarrow \nu(P_{j},x)$ uniformly on $X$, where 
            \begin{align*}
                T^{D}_{j,k} = \ddc u^{D}_{j,k} + (\theta_{j} + \epsilon_k \omega)\\
                T'^{D}_{j,k} = \ddc u'^{D}_{j,k}+(\theta_{j}+\epsilon_{k} \omega)\\
                P^{D}_{j,k} = \ddc p^{D}_{j,k}+ (\theta_{j}+\epsilon_{k} \omega)
            \end{align*}
        \end{itemize}
        
        By Lemma \ref{demaillycompare}, the ordering of the singularity types is preserved after applying Demailly's approximation theorem. To be more precise, for each $k \in \N$, we have $[u^{D}_{j,k}] \succeq [u'^{D}_{j,k}],[p^{D}_{j,k}]$. Since $P_{j}$ is a K\"ahler current for each $j$, there exists $\epsilon>0$ such that $P_{j} \geq \epsilon \omega$ for all $j$. By the construction of $P^{D}_{j,k}$, one sees that $P^{D}_{j,k}$ is also a K\"ahler current and satisfied $P^{D}_{j,k} \geq (\epsilon-\epsilon_{k}) \omega$.

        For $1\leq j \leq m$, set $\lambda_{j,k} \defeq \nu (T^{D}_{j,k},V)$, $\lambda'_{j,k} \defeq \nu(T'^{D}_{j,k},V)$, and let 
        \begin{equation*}
            \delta_{j,k}=(c\|\{\theta_{j}+ \epsilon_{k}\omega\}\|+\frac{\epsilon}{2c_{V}})^{-1}\frac{\lambda'_{j,k}-\lambda_{j,k}}{2},
       \end{equation*}
       where $c_{V},c$ are constants in (\ref{deltachoosing}).
        
        %\[
        %    \delta^{k}_{j}=(c_{1}\|\{\theta+ \epsilon_{k} \omega\}\|+\frac{\epsilon}{2c})^{-1}\frac{\lambda_{j}^{k}-\lambda^{k}}{2},
        %\]
        Now, we apply (\ref{inequ4}) in \textbf{step 1.}, then we obtain
        \begin{align}
                &\int_{X} \langle \wedge_{j=1}^{m} T'^{D}_{j,k} \rangle \wedge \omega^{n-m} \notag\\ 
                \leq
                &\int_{X} \langle \wedge_{j=1}^{m} T^{D}_{j,k} \rangle \wedge \omega^{n-m} -(\lambda'_{m,k}- \lambda_{m,k})  \prod_{j=1}^{m-1}\big{(}\frac{\delta_{j,k} (\epsilon-\epsilon_{k})}{2c_{V}}\big{)}\int_{\widehat{X}}  [\widehat{V}]  \wedge (\sigma^{*}\omega)^{n-m} \wedge \widehat{\omega}^{m-1} \label{inequ5}
        \end{align}

        %\begin{equation}
        %\label{inequ5}
        %    \int_{X} \langle \wedge_{j=1}^{n} T_{j,k}^{D} \rangle \leq \int_{X} \langle (T^{D}_{k})^{n} \rangle -(\lambda^{k}_{n}-\lambda^{k}) \prod_{j=1}^{n-1} \big{(} \frac{\delta^{k}_{j} (\epsilon-\epsilon_{k})}{2c} \big{)}^{n-1} \vol(\widehat{V}).
        %\end{equation}
       By the monotonicity property of non-pluripolar product (Theorem \ref{monotonicity}), we have 
       \begin{equation}
           \int_{X} \langle \wedge_{j=1}^{m} T'_{j} \rangle \wedge \omega^{n-m} \leq \int_{X} \langle \wedge_{j=1}^{m} T'^{D}_{j,k} \rangle \wedge \omega^{n-m} \label{inequ6}.
       \end{equation}
        Since we assume $T_{j}$ is $\mathcal{I}$-model for $j = 1 ,\dots , m$, Lemma \ref{generalizeddarvasconvergencethm} induces 
        \begin{align}  
            \int_{X}\langle  \wedge_{j=1}^{m}T^{D}_{j,k} \rangle \wedge \omega^{n-m}  
            &= \int_{X} \langle \wedge_{j=1}^{m} (\theta_{j}+ \epsilon_{k} \omega)_{u_{j,k}^{D}} \rangle \wedge \omega^{n-m} \notag\\
            &\searrow \int_{X} \langle \wedge_{j=1}^{m}(\theta_{j})_{P[u_{j}]_{\mathcal{I}}} \rangle \wedge \omega^{n-m}= \int_{X} \langle \wedge_{j=1}^{m} T_{j} \rangle \wedge \omega^{n-m}, \mbox{ $k \longrightarrow \infty$}. \label{inequ7}
        \end{align}

        Combining (\ref{inequ5}), (\ref{inequ6}), (\ref{inequ7}) and let $k \longrightarrow \infty$. Then, we get (\ref{inequ4})
        \begin{equation*}
        \label{inequ8}
            \int_{X} \langle \wedge_{j=1}^{m} T'_{j} \rangle  \wedge \omega^{n-m}\leq \int_{X} \langle \wedge_{j=1}^{m}T_{j}\rangle\wedge\omega^{n-m} -(\lambda'_{m}- \lambda_{m})  \prod_{j=1}^{m-1}\big{(}\frac{\delta_{j} \epsilon}{2c}\big{)}^{m-1} \int_{\widehat{X}}  [\widehat{V}]  \wedge (\sigma^{*}\omega)^{n-m} \wedge \widehat{\omega}^{m-1}
        \end{equation*}
        for the general case. Finally, since we assume that $T'_{1}, \dots , T'_{m}$ is of relative full mass intersection with respect to $T_{1}, \dots , T_{m}$, the inequality above does not hold and this make a contradiction. Therefore, $\lambda'_{j}-\lambda_{j}=\nu(T'_{j},V)-\nu(T_{j},V)=0$ for some $j =1 , \dots , m$.
        \\

        \noindent
        \textbf{Step 3.}
        We get rid of the assumption that $V$ is smooth. By using Hironaka's desingularization method, we get $\sigma' \colon X' \longrightarrow X$, which is a composition of finite many blow-ups along smooth centers, such that $V' = \sigma'^{-1}(V)$ is smooth. 
        
        Let $R_{j} = (\sigma')^{*}T_{j}$ and $R'_{j} = (\sigma')^{*}T'_{j}$. In general, $R_{1}', \dots, R'_{m}$ do not necessarily have full mass intersection with respect to $R_{1}, \dots , R_{m}$. However, we have
        \begin{align*}
            \int_{X'} \langle \wedge_{j=1}^{m} R'_{j} \rangle \wedge ((\sigma')^{*}\omega)^{n-m}
            &= \int_{X} \langle \wedge_{j=1}^{m} T'_{j} \rangle \wedge \omega^{n-m}\\ 
            &= \int_{X} \langle \wedge_{j=1}^{m} T_{j} \rangle \wedge \omega^{n-m}\\ 
            &= \int_{X'} \langle \wedge_{j=1}^{m} R_{j} \rangle \wedge((\sigma')^{*}\omega)^{n-m}.
        \end{align*}
        Since $[V] \wedge \omega^{n-m} \neq 0$, it follows that $[V'] \wedge (\sigma'^{*}(\omega))^{n-m} \neq 0$. Hence, we can apply \textbf{step 1.} and \textbf{step 2.} in this setting, and get  
        \[
        \nu(T'_{j},V) = \nu(R'_{j},V')=\nu(R_{j},V') = \nu(T_{j},V)
        \]
        for some $j$. This completes the proof of Theorem \ref{maintheorem}.

\subsection{Proof of Theorem \ref{maintheroem2}}
        Let $\mathscr{B}$ be a closed cone in the cone of big classes, and let $V=\{x_{0}\}$ be a point in $X$. First, note that to prove inequality (\ref{inequa1.2}), it suffices to consider $\{\theta_{j}\} \in \mathscr{S} \cap \mathscr{B}$ for $j = 1 , \dots ,m$, where $\mathscr{S}$ is the unit sphere in $H^{1,1}(X,\R)$. Recall that in (\ref{inequ4}), we choose 
        \begin{equation*}
            \delta_{j}=(c \|\{\theta_{j}\}\|+\frac{\epsilon}{2c_{V}})^{-1}\frac{\lambda'_{j}-\lambda_{j}}{2},
       \end{equation*}
       for $j =1 ,\dots ,m$. Here, $\epsilon$ is derived from the K\"ahler currents $P_{j} \in \{\theta_{j}\}$ such that $P_{j} \geq \epsilon \omega$. Since $\mathscr{S} \cap \mathscr{B}$ is compact, we can choose $\epsilon$ to be independent of $\{\theta_{j}\} \in \mathscr{S} \cap \mathscr{B}$. Also, by the construction of $\widehat{X}$, the constant $c_{V}$ in (\ref{blowupkahler}) is independent of $V=\{x_{0}\}$. Therefore, (\ref{inequ4}) become 
        \begin{align*}
            \int_{X} \langle \wedge_{j=1}^{n}T_{j} \rangle - \langle \wedge_{j=1}^{n} T'_{j} \rangle\geq C\prod_{j=1}^{n} (\nu(T_{j},x_0)-\nu(T,x_{0})),
        \end{align*}
        where $C= \prod_{j=1}^{n-1}(2c \|\theta_{j}\|+ \frac{\epsilon}{c_{V}})^{-1} \vol(\widehat{V})$ is a constant depends only on the cone $\mathscr{B}$ and $X$. This completes the proof.

        \begin{remark}
        \label{relaxrmk}
            One should note that in Theorem \ref{maintheorem} and Theorem \ref{maintheroem2}, the assumption that $T_{j}$ is $\mathcal{I}$-model for $j=1 ,\dots,m$ is too strong. Instead, we only need to assume 
            \[
                \int_{X} \langle \wedge_{j=1}^{m} (\theta_{j})_{P^{\theta_{j}}[u_{j}]_{\mathcal{I}}} \rangle \wedge \omega^{n-m} = \int_{X} \langle \wedge_{j=1}^{m} (\theta_{j})_{u_{j}} \rangle \wedge \omega^{n-m} = \int_{X} \langle \wedge_{j=1}^{m} T_{j} \rangle \wedge \omega^{n-m}.
            \]
           % FIND EXAMPLE!!  u v ARE I-MODEL. BUT u+v IS NOT

             %This assumption is indeed weaker. One could consider $u \in \PSH(X,\theta)$ that is of analytic singularities. Then, we get the singularity types $[u]$ and $[P^{\theta}[u]_{\mathcal{I}}]$ are the same  
        \end{remark}

\bibliography{biblio_family_MA,biblio_Viet_papers}
\bibliographystyle{siam}

\bigskip

\noindent
\Addresses
\end{document}